\documentclass[12pt,a4paper]{amsart}
\usepackage[utf8]{inputenc}%
\usepackage[T1]{fontenc}%
\usepackage[english]{babel}%
\usepackage[margin=2cm]{geometry}%
\usepackage{dsfont}%
\usepackage{graphicx}
\usepackage{amsmath,amssymb,amsthm}%
\usepackage{color}%
\usepackage{float}%
\usepackage{ifthen}%
\usepackage{hyperref}%
\usepackage{tikz}
\usepackage{pgfplots}
\usepackage{mathtools}
\usepackage{booktabs}
\usepackage{pdfpages}
\usepackage[onehalfspacing]{setspace}
\usepackage{amssymb,amsmath,amsfonts}
\usepackage{amsthm}
\usepackage{doi}
\usepackage{hyperref}
\hypersetup{
    hidelinks,
    plainpages=false
}
\usepackage{microtype}
\usepackage{tcolorbox}
\usepackage{enumitem}
\usepackage{longtable}
\usepackage{csquotes}%
\usepackage[
    backend=biber,
    giveninits=true,
    sorting=nyt,
    doi=false,
    url=false,
    isbn=false,
    maxnames=10,
    minnames=1
]{biblatex}
\usepackage[labelformat=simple]{subcaption}

\pgfplotsset{compat=newest}
\usepgfplotslibrary{external} 
\usepackage{mathtools}
\usepgfplotslibrary{colormaps}
\usepgfplotslibrary{groupplots} 
\usetikzlibrary{pgfplots.groupplots} 
\usepackage{multirow}
\usetikzlibrary{decorations.markings}
\pgfplotsset{compat=1.18}

\definecolor{carrotorange}{rgb}{0.93, 0.57, 0.13}
\definecolor{kellygreen}{rgb}{0.3, 0.73, 0.09}
\definecolor{brightcerulean}{rgb}{0.11, 0.67, 0.84} 
\definecolor{cardinal}{rgb}{0.77, 0.12, 0.23} 
\definecolor{tablegray}{gray}{0.95}
	
\def\eff{\mathrm{eff}}
\def\ex{\mathrm{ex}}
\def\ext{\mathrm{ext}}

\def\and{\mathrm{and~}}

\def\trp{\top}

\newcommand{\BDF}[1]{\textnormal{BDF}\ensuremath{(#1)}}

\newcommand{\seq}[3]{\{#1\}_{n=#2}^{#3}}

\newcommand{\noopsort}[1]{}

\def\qq{\qquad}
\def\qqq{\qquad\qq}

\newcommand*\dt{\mathop{}\!\mathrm{d}t}

\newcommand{\traceopD}{\gamma_{\textrm{D}}}
\newcommand{\traceopN}{\gamma_{\textrm{N}}}

 % ...
\newcommand{\R}{\mathbb{R}} 
\renewcommand{\S}{\mathbb{S}} 
\newcommand{\N}{\mathbb{N}} 

\newcommand{\norm}[1]{\lVert #1 \rVert}

\newcommand{\dual}[1]{\langle #1 \rangle}

\def\grad{\boldsymbol{\nabla}}%
\newcommand{\jump}[1]{ [\![#1]\!]_E}

\newcommand{\bLproj}[1]{\ensuremath{\bP_0^{#1}}} 

\def\ee{\mathrm{e}}%
\def\de{\,\mathrm{d}}%
\def\ptl{\partial}%

\renewcommand{\bar}{\overline}
\renewcommand{\hat}{\widehat}
\renewcommand{\tilde}{\widetilde}
\renewcommand{\geq}{\ge}

\theoremstyle{plain}
\newtheorem{theorem}{Theorem}

\newtheorem{Lemma}[theorem]{Lemma}

\newtheorem{Definition}[theorem]{Definition}
\newtheorem{Remark}[theorem]{Remark}

\newtheorem{Assumption}[theorem]{Assumption}
\numberwithin{equation}{section}
\numberwithin{table}{section}
\numberwithin{figure}{section}

\def\AA{{\mathcal A}}

\def\EE{{\mathcal E}}
\def\CC{{\mathcal C}}

\def\II{{\mathcal I}}

\def\PP{{\mathcal P}}
\def\RR{{\mathcal R}}
\def\SS{{\mathcal S}}

\def\QQ{{\mathcal Q}}

\def\TT{{\mathcal T}}
\def\FF{{\mathcal F}}

\def\errM{\boldsymbol{\mathcal{E}}}
\def\errm{\boldsymbol{\epsilon}}
\def\errw{\boldsymbol{\tilde \epsilon}}
\def\errW{\boldsymbol{\tilde{\mathcal{E}}}}
\def\errlam{\varepsilon}
\def\ba{\boldsymbol{a}}%
\def\bb{\boldsymbol{b}}%
\def\bc{\boldsymbol{c}}%
\def\be{\boldsymbol{e}}%
\def\bff{\boldsymbol{f}}%
\def\bH{\boldsymbol{H}}%
\def\bh{\boldsymbol{h}}%
\def\bI{\boldsymbol{I}}%
\def\bL{\boldsymbol{L}}%
\def\bM{\boldsymbol{M}}%
\def\bm{\boldsymbol{m}}%
\def\bn{\boldsymbol{n}}%
\def\bP{\boldsymbol{P}}%
\def\bq{\boldsymbol{q}}%
\def\bRR{\boldsymbol{\RR}}%
\def\br{\boldsymbol{r}}%
\def\bT{\boldsymbol{T}}%
\def\bV{\boldsymbol{V}}%
\def\bv{\boldsymbol{v}}%
\def\bW{\boldsymbol{W}}%
\def\bw{\boldsymbol{w}}%
\def\bx{\boldsymbol{x}}%
\def\bphi{\boldsymbol{\varphi}}%
\def\bPsi{\boldsymbol{\Psi}}%
\def\b0{\boldsymbol{0}}%
\def\Hsym{H}
\def\bHsym{\bH}
\def\Wsym{W}
\def\bWsym{\bW}
\def\Lsym{L}
\def\bLsym{\bL}

\def\bTsym{\bT}

\newcommand{\bWkp}[2]{\bWsym^{#1, #2}}%
\newcommand{\bWspacekp}[2]{\bWsym^{#1, #2}(\Omega)}%

\newcommand{\bLspace}{\bLsym^{2}(\Omega)}
\newcommand{\Hspace}{\Hsym^1(\Omega)}
\newcommand{\bHspace}{\bHsym^1(\Omega)}
\newcommand{\bHk}[1]{\bHsym^{#1}}%

\newcommand{\bLinfspace}{\bLsym^{\infty}(\Omega)}
\newcommand{\Linf}{\Lsym^{\infty}}
\newcommand{\bLinf}{\bLsym^{\infty}}
\newcommand{\spacelam}[1]{V_h^{#1}}%

\newcommand{\etainfty}[1]{\eta(#1; \Linf)}
\newcommand{\etaWinfty}[1]{\eta(#1; \Wsym^{1,\infty})}
\newcommand{\FEMspacep}[1]{\SS_h^{#1}}

\newcommand{\FEMedges}{\Sigma_n}
\newcommand{\FEMinneredges}{\Sigma_n^\circ}

\definecolor{kit-green}{RGB}{0, 150, 130}
\definecolor{kit-green100}{RGB}{0, 150, 130}
\definecolor{kit-green90}{rgb}{0.1, 0.6294, 0.5588}
\definecolor{kit-green80}{rgb}{0.2, 0.6706, 0.6078}
\definecolor{kit-green75}{rgb}{0.25, 0.6912, 0.6324}
\definecolor{kit-green70}{rgb}{0.3, 0.7118, 0.6569}
\definecolor{kit-green60}{rgb}{0.4, 0.7529, 0.7059}
\definecolor{kit-green50}{rgb}{0.5, 0.7941, 0.7549}
\definecolor{kit-green40}{rgb}{0.6, 0.8353, 0.8039}
\definecolor{kit-green30}{rgb}{0.7, 0.8765, 0.8529}
\definecolor{kit-green25}{rgb}{0.75, 0.8971, 0.8775}
\definecolor{kit-green20}{rgb}{0.8, 0.9176, 0.902}
\definecolor{kit-green15}{rgb}{0.85, 0.9382, 0.9265}
\definecolor{kit-green10}{rgb}{0.9, 0.9588, 0.951}
\definecolor{kit-green5}{rgb}{0.95, 0.9794, 0.9755}

\definecolor{kit-blue}{RGB}{70, 100, 170}
\definecolor{kit-blue100}{RGB}{70, 100, 170}
\definecolor{kit-blue90}{rgb}{0.3471, 0.4529, 0.7}
\definecolor{kit-blue80}{rgb}{0.4196, 0.5137, 0.7333}
\definecolor{kit-blue75}{rgb}{0.4559, 0.5441, 0.75}
\definecolor{kit-blue70}{rgb}{0.4922, 0.5745, 0.7667}
\definecolor{kit-blue60}{rgb}{0.5647, 0.6353, 0.8}
\definecolor{kit-blue50}{rgb}{0.6373, 0.6961, 0.8333}
\definecolor{kit-blue40}{rgb}{0.7098, 0.7569, 0.8667}
\definecolor{kit-blue30}{rgb}{0.7824, 0.8176, 0.9}
\definecolor{kit-blue25}{rgb}{0.8186, 0.848, 0.9167}
\definecolor{kit-blue20}{rgb}{0.8549, 0.8784, 0.9333}
\definecolor{kit-blue15}{rgb}{0.8912, 0.9088, 0.95}
\definecolor{kit-blue10}{rgb}{0.9275, 0.9392, 0.9667}
\definecolor{kit-blue5}{rgb}{0.9637, 0.9696, 0.9833}

\definecolor{kit-red}{RGB}{162, 34, 35}
\definecolor{kit-red100}{RGB}{162, 34, 35}
\definecolor{kit-red90}{rgb}{0.6718, 0.22, 0.2235}
\definecolor{kit-red80}{rgb}{0.7082, 0.3067, 0.3098}
\definecolor{kit-red75}{rgb}{0.7265, 0.35, 0.3529}
\definecolor{kit-red70}{rgb}{0.7447, 0.3933, 0.3961}
\definecolor{kit-red60}{rgb}{0.7812, 0.48, 0.4824}
\definecolor{kit-red50}{rgb}{0.8176, 0.5667, 0.5686}
\definecolor{kit-red40}{rgb}{0.8541, 0.6533, 0.6549}
\definecolor{kit-red30}{rgb}{0.8906, 0.74, 0.7412}
\definecolor{kit-red25}{rgb}{0.9088, 0.7833, 0.7843}
\definecolor{kit-red20}{rgb}{0.9271, 0.8267, 0.8275}
\definecolor{kit-red15}{rgb}{0.9453, 0.87, 0.8706}
\definecolor{kit-red10}{rgb}{0.9635, 0.9133, 0.9137}
\definecolor{kit-red5}{rgb}{0.9818, 0.9567, 0.9569}

\definecolor{kit-yellow}{RGB}{252, 229, 0}
\definecolor{kit-yellow100}{RGB}{252, 229, 0}
\definecolor{kit-yellow90}{rgb}{0.9894, 0.9082, 0.1}
\definecolor{kit-yellow80}{rgb}{0.9906, 0.9184, 0.2}
\definecolor{kit-yellow75}{rgb}{0.9912, 0.9235, 0.25}
\definecolor{kit-yellow70}{rgb}{0.9918, 0.9286, 0.3}
\definecolor{kit-yellow60}{rgb}{0.9929, 0.9388, 0.4}
\definecolor{kit-yellow50}{rgb}{0.9941, 0.949, 0.5}
\definecolor{kit-yellow40}{rgb}{0.9953, 0.9592, 0.6}
\definecolor{kit-yellow30}{rgb}{0.9965, 0.9694, 0.7}
\definecolor{kit-yellow25}{rgb}{0.9971, 0.9745, 0.75}
\definecolor{kit-yellow20}{rgb}{0.9976, 0.9796, 0.8}
\definecolor{kit-yellow15}{rgb}{0.9982, 0.9847, 0.85}
\definecolor{kit-yellow10}{rgb}{0.9988, 0.9898, 0.9}
\definecolor{kit-yellow5}{rgb}{0.9994, 0.9949, 0.95}

\definecolor{kit-orange}{RGB}{223, 155, 27}
\definecolor{kit-orange100}{RGB}{223, 155, 27}
\definecolor{kit-orange90}{rgb}{0.8871, 0.6471, 0.1953}
\definecolor{kit-orange80}{rgb}{0.8996, 0.6863, 0.2847}
\definecolor{kit-orange75}{rgb}{0.9059, 0.7059, 0.3294}
\definecolor{kit-orange70}{rgb}{0.9122, 0.7255, 0.3741}
\definecolor{kit-orange60}{rgb}{0.9247, 0.7647, 0.4635}
\definecolor{kit-orange50}{rgb}{0.9373, 0.8039, 0.5529}
\definecolor{kit-orange40}{rgb}{0.9498, 0.8431, 0.6424}
\definecolor{kit-orange30}{rgb}{0.9624, 0.8824, 0.7318}
\definecolor{kit-orange25}{rgb}{0.9686, 0.902, 0.7765}
\definecolor{kit-orange20}{rgb}{0.9749, 0.9216, 0.8212}
\definecolor{kit-orange15}{rgb}{0.9812, 0.9412, 0.8659}
\definecolor{kit-orange10}{rgb}{0.9875, 0.9608, 0.9106}
\definecolor{kit-orange5}{rgb}{0.9937, 0.9804, 0.9553}

\definecolor{kit-lightgreen}{RGB}{140, 182, 60}
\definecolor{kit-lightgreen100}{RGB}{140, 182, 60}
\definecolor{kit-lightgreen90}{rgb}{0.5941, 0.7424, 0.3118}
\definecolor{kit-lightgreen80}{rgb}{0.6392, 0.771, 0.3882}
\definecolor{kit-lightgreen75}{rgb}{0.6618, 0.7853, 0.4265}
\definecolor{kit-lightgreen70}{rgb}{0.6843, 0.7996, 0.4647}
\definecolor{kit-lightgreen60}{rgb}{0.7294, 0.8282, 0.5412}
\definecolor{kit-lightgreen50}{rgb}{0.7745, 0.8569, 0.6176}
\definecolor{kit-lightgreen40}{rgb}{0.8196, 0.8855, 0.6941}
\definecolor{kit-lightgreen30}{rgb}{0.8647, 0.9141, 0.7706}
\definecolor{kit-lightgreen25}{rgb}{0.8873, 0.9284, 0.8088}
\definecolor{kit-lightgreen20}{rgb}{0.9098, 0.9427, 0.8471}
\definecolor{kit-lightgreen15}{rgb}{0.9324, 0.9571, 0.8853}
\definecolor{kit-lightgreen10}{rgb}{0.9549, 0.9714, 0.9235}
\definecolor{kit-lightgreen5}{rgb}{0.9775, 0.9857, 0.9618}

\definecolor{kit-purple}{RGB}{163, 16, 124}
\definecolor{kit-purple100}{RGB}{163, 16, 124}
\definecolor{kit-purple90}{rgb}{0.6753, 0.1565, 0.5376}
\definecolor{kit-purple80}{rgb}{0.7114, 0.2502, 0.589}
\definecolor{kit-purple75}{rgb}{0.7294, 0.2971, 0.6147}
\definecolor{kit-purple70}{rgb}{0.7475, 0.3439, 0.6404}
\definecolor{kit-purple60}{rgb}{0.7835, 0.4376, 0.6918}
\definecolor{kit-purple50}{rgb}{0.8196, 0.5314, 0.7431}
\definecolor{kit-purple40}{rgb}{0.8557, 0.6251, 0.7945}
\definecolor{kit-purple30}{rgb}{0.8918, 0.7188, 0.8459}
\definecolor{kit-purple25}{rgb}{0.9098, 0.7657, 0.8716}
\definecolor{kit-purple20}{rgb}{0.9278, 0.8125, 0.8973}
\definecolor{kit-purple15}{rgb}{0.9459, 0.8594, 0.9229}
\definecolor{kit-purple10}{rgb}{0.9639, 0.9063, 0.9486}
\definecolor{kit-purple5}{rgb}{0.982, 0.9531, 0.9743}

\definecolor{kit-brown}{RGB}{167, 130, 46}
\definecolor{kit-brown100}{RGB}{167, 130, 46}
\definecolor{kit-brown90}{rgb}{0.6894, 0.5588, 0.2624}
\definecolor{kit-brown80}{rgb}{0.7239, 0.6078, 0.3443}
\definecolor{kit-brown75}{rgb}{0.7412, 0.6324, 0.3853}
\definecolor{kit-brown70}{rgb}{0.7584, 0.6569, 0.4263}
\definecolor{kit-brown60}{rgb}{0.7929, 0.7059, 0.5082}
\definecolor{kit-brown50}{rgb}{0.8275, 0.7549, 0.5902}
\definecolor{kit-brown40}{rgb}{0.862, 0.8039, 0.6722}
\definecolor{kit-brown30}{rgb}{0.8965, 0.8529, 0.7541}
\definecolor{kit-brown25}{rgb}{0.9137, 0.8775, 0.7951}
\definecolor{kit-brown20}{rgb}{0.931, 0.902, 0.8361}
\definecolor{kit-brown15}{rgb}{0.9482, 0.9265, 0.8771}
\definecolor{kit-brown10}{rgb}{0.9655, 0.951, 0.918}
\definecolor{kit-brown5}{rgb}{0.9827, 0.9755, 0.959}

\definecolor{kit-cyan}{RGB}{35, 161, 224}
\definecolor{kit-cyan100}{RGB}{35, 161, 224}
\definecolor{kit-cyan90}{rgb}{0.2235, 0.6682, 0.8906}
\definecolor{kit-cyan80}{rgb}{0.3098, 0.7051, 0.9027}
\definecolor{kit-cyan75}{rgb}{0.3529, 0.7235, 0.9088}
\definecolor{kit-cyan70}{rgb}{0.3961, 0.742, 0.9149}
\definecolor{kit-cyan60}{rgb}{0.4824, 0.7788, 0.9271}
\definecolor{kit-cyan50}{rgb}{0.5686, 0.8157, 0.9392}
\definecolor{kit-cyan40}{rgb}{0.6549, 0.8525, 0.9514}
\definecolor{kit-cyan30}{rgb}{0.7412, 0.8894, 0.9635}
\definecolor{kit-cyan25}{rgb}{0.7843, 0.9078, 0.9696}
\definecolor{kit-cyan20}{rgb}{0.8275, 0.9263, 0.9757}
\definecolor{kit-cyan15}{rgb}{0.8706, 0.9447, 0.9818}
\definecolor{kit-cyan10}{rgb}{0.9137, 0.9631, 0.9878}
\definecolor{kit-cyan5}{rgb}{0.9569, 0.9816, 0.9939}

\definecolor{kit-gray}{RGB}{0, 0, 0}
\definecolor{kit-gray100}{RGB}{0, 0, 0}
\definecolor{kit-gray90}{rgb}{0.1, 0.1, 0.1}
\definecolor{kit-gray80}{rgb}{0.2, 0.2, 0.2}
\definecolor{kit-gray75}{rgb}{0.25, 0.25, 0.25}
\definecolor{kit-gray70}{rgb}{0.3, 0.3, 0.3}
\definecolor{kit-gray60}{rgb}{0.4, 0.4, 0.4}
\definecolor{kit-gray50}{rgb}{0.5, 0.5, 0.5}
\definecolor{kit-gray40}{rgb}{0.6, 0.6, 0.6}
\definecolor{kit-gray30}{rgb}{0.7, 0.7, 0.7}
\definecolor{kit-gray25}{rgb}{0.75, 0.75, 0.75}
\definecolor{kit-gray20}{rgb}{0.8, 0.8, 0.8}
\definecolor{kit-gray15}{rgb}{0.85, 0.85, 0.85}
\definecolor{kit-gray10}{rgb}{0.9, 0.9, 0.9}
\definecolor{kit-gray5}{rgb}{0.95, 0.95, 0.95}

\newcommand{\alignedintertext}[1]{%
  \noalign{%
    \vskip\belowdisplayshortskip
    \vtop{\hsize=\linewidth#1\par
    \expandafter}%
    \expandafter\prevdepth\the\prevdepth
  }%
}

\title[A~posteriori error analysis for the \BDF{2} method for the LLG equation]{A~posteriori error analysis for the second-order BDF method for the Landau--Lifshitz--Gilbert equation}
\author[S. Karch]{Stefan Karch}
\address{Institute for Applied and Numerical Mathematics, Karlsruhe Institute of Technology, Englerstr.~2, 76131 Karlsruhe, Germany 
}
\email{stefan.karch@kit.edu}

\begin{document}

\begin{abstract}
    The tangent plane scheme (TPS) is a well-established discretization of the Landau--Lifshitz--Gilbert (LLG) equation.
    However, rigorous \emph{a posteriori} error estimates have not been established. 
    In this work, we derive a rigorous \emph{a posteriori} error estimate for the TPS based on the second-order backward differentiation formula (\BDF{2}) in time and finite elements of arbitrary polynomial degree in space. 
    The proposed estimators provide computable upper bounds for the temporal and spatial discretization errors on adaptive meshes with variable time-step sizes. 
    This result establishes the mathematical foundation for fully adaptive algorithms for the LLG equation.
\end{abstract}

\maketitle

\textbf{Keywords}\quad \def\sep{;\ }
\emph{a~posteriori} error analysis\sep BDF methods\sep energy technique\sep finite elements\sep Landau--Lifshitz--Gilbert equation\sep reconstruction.

\section{Introduction}

The Landau--Lifshitz--Gilbert (LLG) equation is the fundamental model for describing the evolution of the magnetization in ferromagnetic materials. 
Its solutions often exhibit highly localized structures, such as magnetic domain walls, together with rapid temporal changes caused by magnetization switching processes. 
Such localized phenomena arise in a variety of modern magnetic storage technologies, including domain-wall and skyrmion-based racetrack memories, as well as spin-orbit torque magnetic random-access memories
\cite{Schrefl:1999, FertET:2013, HirohataET:2020, FertET:2024}. 
The use of adaptive discretizations is well-suited to resolve these local structures in both space and time.

The analytical theory of the LLG equation is well established. 
Global weak solutions were proved in \cite{Visintin:1985, AlougesET:1992, GuoHong:1993}, although weak solutions are generally not unique \cite{AlougesET:1992}. 
For sufficiently regular initial data, local strong solutions are unique and become global in two space dimensions under suitable smallness assumptions \cite{CarbouFabrie:2001, CarbouFabrie:2001_2}. 
For sufficiently regular initial data, the existence and uniqueness of nontrivial strong solutions with arbitrary spatial and temporal regularity were established in \cite{FeischlTran:2017}.

Numerical discretizations of the LLG equation have been studied extensively. 
For an overview, we refer to the review articles \cite{KurzikProhl:2006, Cimrak:2008}.
Among the available approaches, the tangent plane scheme (TPS) \cite{AlougesJaisson:2006, Alouges:2008} has become one of the standard discretizations of the LLG equation.
While the original TPS formulation employs a nodal projection onto the unit sphere after every time step, several projection-free formulations have recently been proposed \cite{Bartels:2016, AbertET:2014, FeischlTran:2017_2, AFKL:2021, BartelsKovacsWang:2024, AkrivisBartelsPalus:2025, AkrivisET:2025, AldeET:2026}.
These methods avoid the explicit normalization step.
In particular, \cite{AFKL:2021} established optimal \emph{a priori} error estimates for backward differentiation formula (BDF) time discretizations combined with finite elements of arbitrary polynomial degree, together with discrete energy estimates.

The localized nature of LLG solutions has also motivated numerous adaptive numerical strategies, including adaptive mesh refinement, moving mesh methods, and adaptive time-stepping techniques \cite{Bagneres-ViallixET:1991, TakoET:1997, ScholzET:1999, HertelKronmuller:1998, GarciaRoma:2006, Goldenits:2012, Shepherd:2015, ExlET:2017, FangWang:2024, AkrivisET:2025}. 
Nevertheless, the mathematical theory of adaptive methods for the LLG equation remains rather limited. 
Existing \emph{a posteriori} error estimates are restricted to unconstrained or related formulations of the LLG equation \cite{Banas:2005, BanasSlodicka:2006, Banas:2008} and do not apply to the tangent plane formulation with the (weakly) enforced orthogonality constraint. 
To the best of our knowledge, no rigorous \emph{a posteriori} error analysis is available for the TPS.

Building upon the (higher-order) TPS of \cite{AFKL:2021}, we derive rigorous \emph{a posteriori} error estimates for a fully discrete approximation based on the second-order BDF method in time and finite elements of arbitrary polynomial degree in space. 
The proposed estimators provide computable upper bounds for both the spatial and temporal discretization errors on adaptive meshes with variable time-step sizes.
The analysis combines the elliptic reconstruction technique \cite{MakridakisNochetto:2003, LakkisMakridakis:2006} with the three-point time reconstruction \cite{AkrivisChatzipantelidis:2010, LozinskiET:2009}, an approach previously employed, for instance, for the time-dependent Stokes equations in \cite{BaenschBrenner:2019}.
A key advantage of the elliptic reconstruction framework is its natural separation of the spatial and temporal error contributions. 
This allows the reconstruction error to be controlled by established elliptic \emph{a posteriori} error estimators.
The resulting estimates constitute a rigorous \emph{a posteriori} error analysis for a TPS discretization of the LLG equation and establish the mathematical foundation for the design and analysis of fully adaptive algorithms.

We follow the analysis in \cite{AkrivisChatzipantelidis:2010, BaenschBrenner:2019} and first present the estimates for constant time-step sizes for the sake of simplicity, before considering variable time-step sizes.

The remainder of this paper is structured as follows. 
Section~\ref{SEC:NOTATION-PRELIM} first introduces the notation and recalls preliminary results before providing the continuous and fully discrete TPS. 
Subsequently, Section~\ref{SEC:APOST-EST} develops the space--time reconstruction and derives the residual-based \emph{a posteriori} error estimators, establishing the main result in Theorem~\ref{THM: main theorem}. 
Finally, the extension to variable time-step sizes is presented in Section~\ref{SEC:VARIABLE-TS}, yielding the corresponding result in Theorem~\ref{THM: apost variable tau}.

\section{Notation and preliminaries}\label{SEC:NOTATION-PRELIM}

In this section we first recall the continuous LLG equation before establishing the notation required for the subsequent numerical analysis.

\subsection{The Landau--Lifshitz--Gilbert equation}
The magnetic state of a ferromagnet is described by its magnetization $\bm$, represented by a three-dimensional vector field with constant length $|\bm| = 1$.
The classical model to describe the dynamics of the magnetization is given by the Landau--Lifshitz equation 
\begin{align}
    \label{EQ:strong-LL}
    \ptl_t \bm = - \frac{1}{1 + \alpha^2} \bm \times \bh_\eff(\bm) - \frac{\alpha}{1 + \alpha^2} \bm \times (\bm \times \bh_\eff(\bm)),
\end{align}
where $\bh_\eff$ describes the effective field. In this paper, we consider
\begin{align*}
    \bh_\eff(\bm) 
    &= \ell_\ex^2 \Delta \bm 
    + \bh_\ext,
\end{align*} 
where $\ell_\ex$ denotes the exchange length.
Taking the cross product of \eqref{EQ:strong-LL}, adding $\alpha$ times \eqref{EQ:strong-LL} and utilizing the vector identity $\ba \times (\bb \times \bc) = (\ba \cdot \bc) \bb - (\ba \cdot \bb) \bc$ with $|\bm| = 1$, we obtain the equivalent formulation 
\begin{equation}
\label{EQ:alt-LLG}
\begin{aligned}
    \alpha \ptl_t \bm + \bm \times \ptl_t \bm &= \bP(\bm) \bh_\eff(\bm),
    \\
    \ptl_t \bm \cdot \bm &= 0,
\end{aligned}
\end{equation}
where $\bP(\bm) = \bI - \bm \bm^\trp$ denotes the orthogonal projection onto the tangent plane to the unit sphere $\S^2$.
In this paper, we consider \eqref{EQ:alt-LLG} on a bounded domain $\Omega$ and a finite time interval $(0, T)$ with homogeneous Neumann boundary conditions and initial data $\bm^0$ with $|\bm^0| = 1$ in  $\Omega$.

\subsection{Notation}
Let $\Omega \subset \R^d$, $d \in \{2,3\}$, be a bounded domain with Lipschitz boundary.
For $1 \le p \le \infty$ and $k \in \N$, we write $\Lsym^p(\Omega)$ and $\Wsym^{k,p}(\Omega)$ for the usual Lebesgue and Sobolev spaces, respectively, and use the abbreviation $\Hsym^k(\Omega) \coloneqq \Wsym^{k,2}(\Omega)$.
The corresponding Bochner spaces are denoted analogously.
The associated norms are denoted by the corresponding subscripts, and the domain is omitted whenever it is clear from the context.
For convenience, we abbreviate the $\Lsym^2(\Omega)$-norm by $\norm{\cdot}\coloneqq\norm{\cdot}_{\Lsym^2(\Omega)}$.
Throughout, boldface symbols indicate vector-valued quantities.
We set $V \coloneqq \Hspace$, $\bV \coloneqq V^3$ and denote by $V'$ the dual of $V$.
We denote by $\dual{\cdot, \cdot}$ the duality pairing between $V'$ and $V$, which coincides with the $L^2$-inner product whenever both arguments belong to $L^2(\Omega)$.
Furthermore, we denote by $V_h^n \coloneqq \FEMspacep{p}(\TT_h^n) \subset V$ the conforming finite element space of polynomial degree $p \ge 1$ associated with the shape-regular triangulation $\TT_h^n$ of $\Omega$ at time $t_n$, $n \ge 0$.
Finally, we denote by $\Sigma_n$ the set of all edges (for $d=2$) or faces (for $d=3$) of $\TT_h^n$ and by $\Sigma_n^\circ$ the set of all interior edges or faces.

Let $0=t_0<t_1<\dots<t_N=T$ be a partition of $[0, T]$ with time-steps $\tau_n \coloneqq t_n - t_{n-1}$. 
Let $\bLproj{n}: \bL^2(\Omega)\to\bV_h^n$ denote the $\bL^2$-orthogonal projection onto the finite element space $\bV_h^n$. For a sequence
$\{\bm_h^n\}_{n=-1}^N$, we define the first-order backward difference by
\begin{align}
    \bar\partial \bm_h^n
    \coloneqq
    \bar\partial^1 \bm_h^n
    \coloneqq
    \frac{\bm_h^n-\bm_h^{n-1}}{\tau_n}
    =
    \frac{\bm_h^n-\bm_h^{n-1}}{t_n-t_{n-1}},
    \qquad n\ge0.
\end{align}
Higher-order backward differences are defined recursively by
\begin{align}
    \label{EQ: Def ptl bar k}
    \bar\partial^k \bm_h^n
    \coloneqq
    \bar\partial\big(\bar\partial^{k-1}\bm_h^n\big),
    \qquad
    2\le k\le n+1.
\end{align}
Since the finite element spaces may vary in time, the quantities
$\bar\partial^k\bm_h^n$ do not, in general, belong to $\bV_h^n$.
We therefore introduce the projected backward differences
\begin{align*}
    \bar\partial_n^k\bm_h^n
    \coloneqq
    \bLproj{n}\bar\partial^k\bm_h^n
    \in\bV_h^n,
    \qquad
    1\le k\le n+1.
\end{align*}

For $n\ge1$, we define the piecewise linear reconstruction
\begin{align}
    \label{EQ: lin. reconstruction}
    \bm_h(t)
    \coloneqq
    \bm_h^n
    +(t-t_n)\bar\partial\bm_h^n,
    \qquad
    t\in I_n\coloneqq(t_{n-1},t_n].
\end{align}
Moreover, for $n\ge2$, we define the quadratic three-point reconstruction
\begin{equation}
\label{EQ:def-Mh_uni}
\begin{aligned}
    \bM_h(t)
    &\coloneqq
    \bm_h(t)
    +\frac12(t-t_n)(t-t_{n-1})\bar\partial^2\bm_h^n,\\
    \partial_t\bM_h(t)
    &=
    \bar\partial^B\bm_h^n
    +(t-t_n)\bar\partial^2\bm_h^n,
\end{aligned}
\end{equation}
where
\begin{align}
    \label{EQ: Def ptlB}
    \bar\partial^B\bm_h^n
    \coloneqq
    \bar\partial\bm_h^n
    +\frac{\tau}{2}\bar\partial^2\bm_h^n,
    \qquad
    \bar\partial_n^B\bm_h^n
    \coloneqq
    \bLproj{n}\bar\partial^B\bm_h^n.
\end{align}
The operator $\bar\partial^B$ coincides with the BDF(2) time discretization operator.

Finally, on the initial interval $I_1=(t_0,t_1]$, we define
\begin{align}
    \label{EQ: initial tp-reconstruction}
    \bM_h(t)
    =
    \bm_h(t)
    +\frac12(t-t_0)(t-t_1)\bar\partial^2\bm_h^2,
    \qquad
    \partial_t\bM_h(t)
    =
    \bar\partial\bm_h^1
    +(t-t_{1/2})\bar\partial^2\bm_h^2,
\end{align}
where $t_{1/2}\coloneqq (t_0+t_1)/2$. By a slight abuse of notation, the symbol
$\bM_h$ is also used for this initial reconstruction, but its meaning is clear from the context.

\subsection{Preliminaries}

We define the bilinear form $a: \bV \times \bV \to \R$ by
\begin{align}
    \label{EQ: bilinform a}
    a(\bm, \bphi) = \dual{\grad \bm, \grad \bphi}.
\end{align}
We first introduce the discrete Laplacian for functions in $\bHk{1}(\Omega)$.
An analogous definition for functions in $\bH_0^1(\Omega)$ is given in \cite[Def.~1.1]{BaenschKarakatsaniMakridakis:2013}.

\begin{Definition}[Discrete Laplacian]
    \label{DEF: Discrete laplacoan}
    Let $\Omega \subset \R^d$, $d=2,3$, be a bounded domain with Lipschitz boundary and $\bv \in \bHspace$.
    Then the discrete Laplacian $- \Delta_h^n \bv \in \bV_h^n \subset \bHspace$ is the operator with the property 
    \begin{align}
        \label{EQ: discrete Laplacian def}
        \dual{- \Delta_h^n \bv, \bphi_h^n} = a(\bv, \bphi_h^n) - \dual{\traceopN \bv, \traceopD \bphi_h^n} 
        , \qq \forall \bphi_h^n \in \bV_h^n,
    \end{align}
    where $\traceopD$ and $\traceopN$ denote the Dirichlet and Neumann trace operators,~cf.~\cite[Theorem~3.37]{McLean:2000} and \cite[Lemma~4.3]{McLean:2000}, respectively. 
\end{Definition}
We proceed with the definition of the elliptic reconstruction associated with the bilinear form $a$ \eqref{EQ: bilinform a} and the finite element space $\bV_h^n$ \cite{LakkisMakridakis:2006, MakridakisNochetto:2003}.

\begin{Definition}[Elliptic reconstruction]
    \label{DEF: Elliptic reconstruction}
    We define for $\bm_h^n \in \bV_h^n \subset \bHspace$ the elliptic reconstruction \hbox{$\bRR \bm_h^n \in \bHspace$} by
    \begin{align}
        \label{EQ: elliptic reconstruction}
        a({\bRR \bm_h^n, \bphi}) = \dual{ - \Delta_h^n \bm_h^n, \bphi}, \qq \forall \bphi \in \bHspace.
    \end{align}
\end{Definition}

A main characteristic of the elliptic reconstruction is that we can use any available a posteriori estimate for the elliptic part of the equation \cite{LakkisMakridakis:2006}. 

\begin{Assumption}
    Let $\bm_h^n \in \bV_h^n$ and $\bRR \bm_h^n \in \bV$ be the elliptic reconstruction defined in \eqref{EQ: elliptic reconstruction}. Then, we assume that there exist a posteriori error estimators $\eta(\bm_h^n; \Lsym^2)$ and $\eta(\bm_h^n; \Hsym^1)$, such that
    \begin{align}
        \label{EQ: ell. rec. a post}
        \norm{(\bI - \bRR) \bm_h^n}^2 \lesssim \eta(\bm_h^n; \Lsym^2),
        \qq
        \norm{(\bI - \bRR) \bm_h^n}_{\bHk{1}}^2 \lesssim \eta(\bm_h^n; \Hsym^1).
    \end{align}
\end{Assumption}

A possible choice for the estimators is given by
\begin{equation}
\label{EQ: eta estimator standard}
\begin{aligned}
    \eta(\bm_h^n;  \Lsym^2) \coloneqq \sum\limits_{K \in \TT_h^n} h_K^4 \norm{R_K(\bm_h^n)}_{\bLsym^2(K)}^2 + \sum\limits_{E \in \FEMedges} h_E^3 \norm{R_E(\bm_h^n)}_{\bLsym^2(E)}^2,
    \\
    \eta(\bm_h^n; \Hsym^1) \coloneqq \sum\limits_{K \in \TT_h^n} h_K^2 \norm{R_K(\bm_h^n)}_{\bLsym^2(K)}^2 + \sum\limits_{E \in \FEMedges} h_E \norm{R_E(\bm_h^n)}_{\bLsym^2(E)}^2,
\end{aligned}
\end{equation}
where the element and edge residuals are defined by 
\begin{align*}
    R_K(\bm_h^n) \coloneqq (\Delta - \Delta_h^n)\bm_h^n
\end{align*}
for $K \in \TT_h^n$ and 
\begin{align*}
    R_E(\bm_h^n) \coloneqq 
    \begin{cases}
    \jump{\ptl_{\bn} \bm_h^n}, & \text{for } E \in \FEMinneredges, \\
    \ptl_{\bn} \bm_h^n, & \text{for } E \in \FEMedges \cap \ptl \Omega.
    \end{cases}
\end{align*}
For details, we refer the interested reader to \cite[Theorem~3.1]{AinsworthOden:1997}, \cite[Prop.~3.8]{Verfuerth:1996} and \cite{Verfuerth:2013}.

We also require a~posteriori bounds in the $\bLsym^\infty$-norm for the elliptic reconstruction.
Again, since the elliptic reconstruction satisfies an associated elliptic problem, any existing a~posteriori $\bLsym^\infty$-estimate transfers directly to the reconstruction.
For instance, Nochetto~et~al. derived a posteriori error estimates in the $\Linf$-norm for monotone semi‐linear elliptic problems on a bounded, polyhedral domain in \cite[Theorem~4.2]{NochettoET:2006}.
Similarly, in \cite[Theorem~3.1]{Demlow:2006}, an a posteriori bound in the $\Linf$-norm of the gradient of the error of piecewise linear finite element approximations is established on convex polyhedral domains for quasilinear elliptic problems.

\begin{Assumption}
    \label{AS: Linfty estimators}
    Let $\bm_h^n \in \bV_h^n$ and $\bRR \bm_h^n \in \bWspacekp{1}{\infty}$ be the elliptic reconstruction defined in \eqref{EQ: elliptic reconstruction}. Then, we assume that there exist a posteriori error estimators $\etainfty{\bm_h^n}$ and $\etaWinfty{\bm_h^n}$, such that
    \begin{align}
        \label{EQ: apost ell rec Linfty}
        \norm{\bRR \bm_h^n - \bm_h^n}_{\bLinf}^2 &\lesssim \etainfty{\bm_h^n},
        \qq 
        \norm{\grad(\bRR \bm_h^n - \bm_h^n)}_{\bLinf}^2 \lesssim \etaWinfty{\bm_h^n}.
    \end{align}
\end{Assumption}

Note that we may have to impose additional assumptions on the domain $\Omega$ to ensure the validity of the a posteriori estimates \eqref{EQ: ell. rec. a post} and \eqref{EQ: apost ell rec Linfty},~e.g.,~smoothness of the boundary or convexity.
We emphasize that, for the derivation of the optimal-order a~posteriori error estimators later in Theorem~\ref{THM: main theorem}, no convergence properties of the estimators $\etainfty{\bm_h^n}$ and $\etaWinfty{\bm_h^n}$ are required.
Instead, it suffices to assume the uniform bound $\etaWinfty{\bm_h^n} \le C$, where $C>0$ is independent of the discretization parameters.

Further, we require a Lipschitz-type bound for the projection onto the tangent plane. 
The following Lemma is an extension of \cite[Lemma~4.1]{AFKL:2021} to the case $|\bm| \neq 1$.

\begin{Lemma}[Lipschitz-type bound]
    \label{LEM: Lipschitz bound}
    Let $\tilde \bm, \bv \in \bHspace \cap \bWspacekp{1}{\infty}$ and $\bm \in \bHspace \cap \bLinfspace$.
    Then, the projection
    $\bP(\bm) = \bI - \bm \bm^\trp$ satisfies
    \begin{align*}
        \norm{(\bP(\tilde \bm) - \bP(\bm))\bv}
        &\le
        \big(\norm{\bm}_{\bLinf} + \norm{\tilde \bm}_{\bLinf} \big) \norm{\bv}_{\bLinf} \norm{\tilde \bm - \bm},
        \\
        \norm{\grad (\bP(\tilde \bm) - \bP(\bm))\bv}
        &\le 
        \big( 2 \norm{\bv}_{\bLinf} \norm{\grad \tilde \bm}_{\bLinf} + \norm{\grad \bv}_{\bLinf} (\norm{\bm}_{\bLinf} + \norm{\tilde \bm}_{\bLinf}) \big)\norm{\tilde \bm - \bm}        
        \\&\quad
        + 2 \norm{\bv}_{\bLinf} \big(\norm{\bm}_{\bLinf} + 2 \norm{\tilde \bm}_{\bLinf} \big) \norm{\grad(\tilde \bm - \bm)}
        \\&\le \gamma(\bm, \tilde \bm, \bv) \norm{\tilde \bm - \bm}_{\bHk{1}},
    \end{align*}
    where $\gamma$ is defined by
    \begin{align*}
        \gamma(\bm, \tilde \bm, \bv) \coloneqq 
         2 \norm{\bv}_{\bLinf} (\norm{\bm}_{\bLinf} + 2 \norm{\tilde \bm}_{\bLinf} + \norm{\grad \tilde \bm}_{\bLinf}) + \norm{\grad \bv}_{\bLinf} (\norm{\bm}_{\bLinf} + \norm{\tilde \bm}_{\bLinf}).
    \end{align*}
    If $\norm{\bm}_{\bLinf} = 1$, then $\gamma(\bm, \tilde \bm, \bv) = \gamma(\tilde \bm, \bv)$.
\end{Lemma}
\begin{proof}
    Setting $\be = \bm - \tilde \bm$, we rewrite the difference of the projections as
    \begin{align*}
        (\bP(\tilde \bm) - \bP(\bm))\bv
        = (\bm \be^\trp + \be \tilde \bm^\trp) \bv.
    \end{align*}
    Differentiating componentwise and substituting $\bm = \be + \tilde \bm$ yields
    \begin{align*}
        \ptl_i (\bP(\tilde \bm) - \bP(\bm))\bv
        &= (\ptl_i \be \be^\trp + \be \ptl_i \be^\trp + \ptl_i \tilde \bm \be^\trp + \tilde \bm \ptl_i \be^\trp + \ptl_i \be \tilde \bm^\trp + \be \ptl_i \tilde \bm^\trp)\bv 
        \\&\quad + (\bm \be^\trp + \be \tilde \bm^\trp) \ptl_i \bv.
    \end{align*}
    Estimating the terms using H\"older's inequality along with the bound $\norm{\be}_{\bLinf} \le \norm{\bm}_{\bLinf} + \norm{\tilde \bm}_{\bLinf}$ completes the proof.
\end{proof}

\subsection{Weak formulation}

For the sake of simplicity, we substitute the external field $\bh_\ext$ with $\bff$. 
Moreover, since the notation is already quite involved, we set $\ell_\ex^2 = 1$. 
However, the proofs can be generalized directly to arbitrary $\ell_\ex^2 > 0$.

For $\bm \in \bV$, we define the tangent space
\begin{align*}
    \bTsym(\bm) = \{ \bphi \in \bV \,:\, \bm \cdot \bphi = 0~\textrm{a.e.}\} = \{ \bphi \in \bV\,:\, \bP(\bm) \bphi = \bphi\}.
\end{align*}
Following the tangent plane approach of \cite{AlougesJaisson:2006, Alouges:2008}, we consider the weak formulation:
Find $\bm \in \Lsym^2(0, T; \bV)$ with $\ptl_t \bm \in \Lsym^2(0, T; \bTsym(\bm))$ such that
\begin{align}
	\label{EQ: linearized LLG analytical}
	{\alpha} \dual{\ptl_t \bm, \bphi} + \dual{\bm \times \ptl_t \bm, \bphi} + a(\bm, \bphi) = \dual{\bff, \bphi}, \qq \forall \bphi \in \Lsym^2(0, T; \bTsym(\bm)).
\end{align}
Introducing a Lagrange multiplier to enforce the normalization constraint, the above formulation is equivalently written as the following saddle point problem:
Find $(\bm, \lambda) \in \Hsym^1(0, T; \bV) \times \Lsym^2(0, T; V')$ 
\begin{equation}
	\label{EQ: linearized LLG analytical SPP}
	\begin{aligned}
		{\alpha} \dual{\ptl_t \bm, \bphi} + \dual{\bm \times \ptl_t \bm, \bphi} + a(\bm, \bphi) + b^{\bm}(\lambda, \bphi) &= \dual{\bff, \bphi}, 
		\\
		b^{\bm}(\psi, \ptl_t \bm) &= 0
	\end{aligned}
\end{equation}
for all $(\bphi, \psi) \in \Lsym^2(0, T;\bV) \times \Lsym^2(0, T; V')$, where $b^{\bm}$ is defined by
\begin{align}
    \label{EQ: Def b}
    b^{\bm}(\lambda, \bphi) \coloneqq \dual{\lambda, \bm \cdot \bphi}.
\end{align}

The well-posedness of the saddle point formulation \eqref{EQ: linearized LLG analytical SPP} is given by the following inf-sup condition.

\begin{Lemma}[Inf-sup condition]\label{LEM:inf-sup}
	Let $\bm \in \bWspacekp{1}{\infty}$ with $|\bm| = 1$ almost everywhere. Then, 
	for $b^{\bm} : V' \times \bV \to \R$ defined in \eqref{EQ: Def b}, 
	the inf-sup condition
	\begin{align*}
		\beta(\bm) \norm{\lambda}_{V'}
		\le 
		\sup_{\bphi \in \bV \setminus \{\b0\}}
		\frac{b^{\bm}(\lambda, \bphi)}{\norm{\bphi}_{\bV}}
		=
		\sup_{\bphi \in \bV \setminus \{\b0\}}
		\frac{\dual{\lambda, \bm\cdot\bphi}}{\norm{\bphi}_{\bV}}
	\end{align*}
	holds, where $\beta$ is defined by
	\begin{align*}
		\beta(\bm) \coloneqq \frac{1}{\sqrt{3}\max\big\{1, \norm{\grad\bm}_{\bLsym^\infty}\big\}}.
	\end{align*}
\end{Lemma}
\begin{proof}
	Using $\bm \cdot \bm =1$ a.e. leads to
	\begin{align*}
		\norm{\lambda}_{V'}
		&= \sup_{\zeta\in V\setminus\{0\}}
		\frac{\dual{\lambda,\zeta}}{\norm{\zeta}_{V}}
		= \sup_{\zeta\in V\setminus\{0\}}
		\frac{\dual{\lambda,\bm\cdot\zeta\bm}}%
		{\norm{\zeta}_{V}}.
	\end{align*}
	Since $\bm \in \bWspacekp{1}{\infty}$ and $|\bm| = 1$ a.e., we estimate with Young's inequality
	\begin{equation}
		\label{EQ: inf-sup beta deriv}
		\begin{aligned}
			\norm{\zeta \bm}_{\bV}^{2} 
			&= \int_\Omega \big( |\zeta \bm|^2 + | \grad (\zeta \bm) |^2 \big) \de \bx
			\\&\le
			\int_\Omega \big( |\zeta|^2 |\bm|^2 + 2 |\zeta|^2|\grad\bm|^2
			+ 2|\grad\zeta|^2|\bm|^2\big) \de \bx
			\\&\le 
			3\max\big\{1, \norm{\grad\bm}_{\bLsym^\infty}^{2}\big\}\norm{\zeta}_{V}^{2}
			= \frac1{\beta(\bm)^2}\norm{\zeta}_{V}^{2}.
		\end{aligned}
	\end{equation}
	Finally, applying \eqref{EQ: inf-sup beta deriv} we conclude
	\begin{align*}
		\norm{\lambda}_{V'}
		= 
		\sup_{\zeta\in V\setminus\{0\}} \frac{\dual{\lambda,\bm\cdot\zeta\bm}}{\norm{\zeta}_{V}}
		\le 
		\frac1{\beta(\bm)} \sup_{\zeta\in V\setminus\{0\}} \frac{\dual{\lambda,\bm\cdot\zeta\bm}}{\norm{\zeta\bm}_{V}{}}
		\le \frac1{\beta(\bm)}\sup_{\bphi\in\bV\setminus\{\b0\}} \frac{\dual{\lambda,\bm\cdot\bphi}}{\norm{\bphi}_{\bV}}.
	\end{align*}
\end{proof}

\begin{Remark}
    The regularity assumption $\bm\in\bWspacekp{1}{\infty}$ in Lemma~\ref{LEM:inf-sup} is only imposed for simplicity. 
    The proof merely requires that multiplication by $\bm$ defines a bounded operator from $V$ to $\bV$. 
    By utilizing Hölder's inequality alongside the Sobolev embedding $\Hsym^1(\Omega) \hookrightarrow \Lsym^{2p/(p-2)}(\Omega)$, the gradient bound remains valid under the weaker assumption $\bm\in \bWspacekp{1}{p}$ for some $p>d$, where $d$ denotes the spatial dimension.
    In this relaxed case, the inf-sup constant $\beta(\bm)$ depends continuously on the $\bLsym^p$-norm of the gradient $\|\grad\bm\|_{\bLsym^p(\Omega)}$.
\end{Remark}

\subsection{Model equation and full discretization}
We consider the discrete tangent space in a weak $\Lsym^2$-sense as in \cite{AFKL:2021}.
Namely, for a function $\bm \in \bHspace$, we define the discrete tangent space as
\begin{align}
    \label{EQ: discrete tps bThn apost}
    \bTsym_h^n(\bm) = \big\{ \bphi_h \in \bV_h^n \,:\, \dual{\bm \cdot \bphi_h, \psi_h} = 0, \quad \forall \psi_h \in \spacelam{n} \big\}.
\end{align}
Note that $\bTsym_h^n(\bm) \not \subset \bTsym(\bm)$ in general.
The discrete tangent space we use differs from that used in the works \cite{AlougesJaisson:2006, Alouges:2008}, where the  normalization constraint is satisfied at each node of the finite element mesh.
Then, the fully discrete scheme employing the \BDF{2} scheme reads as follows: Find $\bm_h^n$ with $\bv_h^n \coloneqq \bar \ptl^B \bm_h^n \in \bTsym_h^n(\hat \bm_h^n)$ such that
\begin{align}
    \label{EQ: discrete LLG}
    {\alpha} \dual{\bv_h^n, \bphi_h^n} + \dual{\hat \bm_h^n \times \bv_h^n, \bphi_h^n} + a(\bm_h^n, \bphi_h^n) = \dual{\bff_h^n, \bphi_h^n}, \qq \forall \bphi_h^n \in \bTsym_h^n(\hat \bm_h^n),
\end{align}
where $\bff_h^n \coloneqq \bLproj{n} \bff(t_n)$ and $\hat \bm_h^n$ is a predictor for $\bm_h^n$. 
Note that our analysis does not rely on any specific structure of $\hat \bm_h^n$ 
and only assumes that $\hat \bm_h^n \in \bV$, i.e., $\hat \bm_h^n \not \in \bV_h^n$ is admissible.
In general, we first compute an extrapolation $\bq_h^n$ for $\bm_h^n$ from a composition of previously computed values $\bm_h^{n-1}, \dots, \bm_h^{n-j}$, $j \ge 1$. 
Preferably, we then normalize the approximation $\bq_h^n$ to employ the predictor $\hat \bm_h^n = \bq_h^n / |\bq_h^n|$ as in \cite[(2.1)]{AFKL:2021}.
However, this normalization leads to $\hat \bm_h^n \not \in \bV_h^n$.
In practice, we may only normalize the predictor at nodal values and thus ensure that the predictor remains in $\bV_h^n$.

Equivalently, we formulate the corresponding saddle point problem by seeking $(\bm_h^n, \lambda_h^n) \in \bV_h^n \times \spacelam{n} \subset \bV \times V'$ such that
\begin{equation}
    \label{EQ: discrete LLG SPP}
    \begin{aligned}
        {\alpha} \dual{\bv_h^n, \bphi_h^n} + \dual{\hat \bm_h^n \times \bv_h^n, \bphi_h^n} + a(\bm_h^n, \bphi_h^n) + b^{\hat \bm_h^n}(\lambda_h^n, \bphi_h^n) &= \dual{\bff_h^n, \bphi_h^n}, 
        \\
        b^{\hat \bm_h^n}(\psi_h^n, \bv_h^n) &= 0
    \end{aligned} 
\end{equation}
for all $(\bphi_h^n, \psi_h^n) \in \bV_h^n \times \spacelam{n}$, where $b^{\bm}$ is defined in \eqref{EQ: Def b}.
The well-posedness of the fully discrete scheme \eqref{EQ: discrete LLG SPP} is ensured by the discrete inf-sup conditions as provided in \cite[p.~1013]{AFKL:2021}.

For the initialization of the numerical scheme, we require the given initial condition $\bm^0 \in \bHspace$ with $|\bm^0| = 1$ a.e. in $\Omega$. 
Then, we set $\bm_h^0 \coloneqq \bLproj{0} \bm^0$.
Further, we require $\lambda_h^0$ to compute the first time step later using the trapezoidal scheme \eqref{EQ: init trapezoidal m1}.
We determine $(\bv_h^0, \lambda_h^0) \in \bV_h^0 \times \spacelam{0}$ by computing
\begin{equation}
    \label{EQ: init v0 lambda0}
    \begin{aligned}
        \alpha \dual{\bv_h^0, \bphi_h^0} + \dual{\bm_h^0 \times \bv_h^0, \bphi_h^0} + b^{\bm_h^0}(\lambda_h^0, \bphi_h^0) &= \dual{\bff_h^0, \bphi_h^0} - a(\bm_h^0, \bphi_h^0), 
        \\
        b^{\bm_h^0}(\psi_h^0, \bv_h^0) &= 0,
    \end{aligned}
\end{equation}
for all $(\bphi_h^0, \psi_h^0) \in \bV_h^0 \times \spacelam{0}$.
Setting $\hat \bm_h^0 = \bm_h^0$, we treat the first time step using the trapezoidal rule by finding $(\bm_h^1, \lambda_h^1) \in \bV_h^1 \times \spacelam{1}$ such that
\begin{equation}
    \label{EQ: init trapezoidal m1}
    \begin{aligned}
        {\alpha} \dual{\bar \ptl \bm_h^1, \bphi_h^1} + \dual{\hat \bm_h^{1/2} \times \bar \ptl \bm_h^1, \bphi_h^1} + a(\bm_h^{1/2}, \bphi_h^1) + b^{\hat \bm_h^{1/2}}(\lambda_h^{1/2}, \bphi_h^1) &= \dual{\bff_h^{1/2}, \bphi_h^1}, 
        \\
        b^{\hat \bm_h^{1/2}}(\psi_h^1, \bar \ptl \bm_h^1) &= 0,
    \end{aligned}
\end{equation}
for all $(\bphi_h^1, \psi_h^1) \in \bV_h^1 \times \spacelam{1}$, 
where the superscript $1/2$ denotes the average of the values at $t_0$ and $t_1$, e.g., $\bff_h^{1/2} = \frac12 (\bff_h^1 + \bff_h^0)$.
Possible choices for the predictor $\hat \bm_h^1$ are the normalized first-order approximation $\hat \bm_h^1 = \hat \bm_h^0 / |\hat \bm_h^0|$ or the normalized second-order approximation $\hat \bm_h^1 = (\hat \bm_h^0 + \tau \bv_h^0) / |\hat \bm_h^0 + \tau \bv_h^0|$.

\section{\emph{A~posteriori} estimate for the fully discrete LLG equation}\label{SEC:APOST-EST}

Before presenting the main result of this chapter, we first introduce the a~posteriori error estimators that we obtain later in our analysis. 

\begin{Definition}[A~posteriori error estimators]
\label{DEF: error est.}
Let $\seq{\bm_h^n}{0}{N}$, $\seq{\hat \bm_h^n}{0}{N}$, $\seq{\bff_h^n}{0}{N}$, $\seq{\lambda_h^n}{0}{N}$ be sequences with $\bm_h^n, \bff_h^n \in \bV_h^n$, $\lambda_h^n \in \spacelam{n}$ and $\hat \bm_h^n \in V$ for all $n \ge 0$
and define for convenience,~cf.~Lemma~\ref{LEM: apost rh and rh0},
\begin{align}
    \label{EQ: mh -1}
    {\alpha} \bm_h^{-1} &\coloneqq {\alpha} \bm_h^1 - 2 \tau \big( \bff_h^0 + \Delta_h^0 \bm_h^0 - \bLproj{1}(\hat \bm_h^0 \times \bar \ptl \bm_h^1) - \bLproj{1} (\lambda_h^0 \hat \bm_h^0) \big).
\end{align}
Then, we introduce the \textbf{time error estimators}
\begin{equation}
    \begin{aligned}
        &\EE_1 \coloneqq \max\limits_{2 \le n \le N} \norm{\grad \bar \ptl^2_n \bm_h^n}^2, 
        &&\EE_2 \coloneqq \sum\limits_{n=2}^N \frac{1}{\tau} \norm{\bar \ptl_n^2 \bm_h^n - \bar \ptl_{n-1}^2 \bm_h^{n-1}}^2,
        \\
        &\EE_3 \coloneqq \sum\limits_{n=1}^N \tau \big(\norm{\bar \ptl \bm_h^n \bar \ptl \lambda_h^n}^2 + \norm{\bar \ptl \hat \bm_h^n \bar \ptl \lambda_h^n}^2 \big), 
        &&\EE_6 \coloneqq \sum\limits_{n=2}^N \tau \norm{\bar \ptl^2 \bm_h^n}^2,
        \\
        \mathclap{\hspace{15cm}\EE_4\coloneqq \sum\limits_{n = 2}^N \tau \Big(\norm{\bar \ptl \bm_h^n \times \bar \ptl^2 \bm_h^n}^2 + \norm{\bar \ptl \hat \bm_h^n \times \bar \ptl^2 \bm_h^n}^2 +\norm{\hat \bm_h^{n-1} \times \bar \ptl^3 \bm_h^n}^2\Big),}
    \end{aligned}
\end{equation}
the \textbf{reconstruction error estimator}
\begin{align*}
    \EE_5 \coloneqq 
    \sum\limits_{n = 2}^N \tau \norm{\bar \ptl^2 (- \Delta_h^n)\bm_h^n}^2,
\end{align*}
the \textbf{projection error estimator}
\begin{align*}
    &\PP \coloneqq \int_{0}^{t_N} \norm{(\bI - \bP(\bM_h)) \ptl_t \bM_h}_{\bHk{1}}^2 \dt,
\end{align*}
where $\bM_h$ is defined in \eqref{EQ:def-Mh_uni} - \eqref{EQ: initial tp-reconstruction},
the \textbf{space error estimators}
\begin{align*}
    &\Lambda_1 \coloneqq  \max\limits_{0 \le n \le N} \, \eta(\bm_h^n; \Hsym^1), \qq \Lambda_2 \coloneqq \sum\limits_{n=1}^N \tau\,\eta(\bar\ptl_n \bm_h^n; \Hsym^1),
    \\
    &\qqq \qq\Lambda_3 \coloneqq \sum\limits_{n=0}^N \tau\,\eta(\bm_h^n; \Lsym^2),
\end{align*}
where $\eta$ is the elliptic a~posteriori estimator defined in \eqref{EQ: eta estimator standard}, the \textbf{data approximation error estimators}
\begin{align*}
    \FF_1 \coloneqq \int_{I_1} 
        \norm{\bff_h - \bff}^2 
        \dt, \qq
    \FF_2 \coloneqq \sum\limits_{n=2}^N \int_{I_n}
        \norm{\bff_h - \bff}^2 
        \dt, 
\end{align*}
the \textbf{changing mesh estimators}
\begin{align*}
    \Xi_1 &\coloneqq \sum\limits_{n=2}^N \tau \norm{(\bI - \bRR)(\bar \ptl - \bar \ptl_n) \bm_h^n}^2, 
    \qq \Xi_2 \coloneqq \tau^4 \max\limits_{2 \le n \le N}\norm{\grad (\bar \ptl^2 - \bar \ptl^2_n) \bm_h^n}^2,
    \\
    \Xi_3 &\coloneqq 
    \sum\limits_{n=2}^N 
    \Big(  
    \tau \norm{(\bar \ptl^B - \bar \ptl_n^B) \bm_h^n}^2
    + \tau^3 \norm{\bar \ptl \big( (\bar \ptl - \bar \ptl_n) \bm_h^n \big)}^2 
    \Big),
\end{align*}
the \textbf{finite element space conforming estimators}
\begin{align*}
    \CC_1 &\coloneqq \tau \norm{(\bI - \bLproj{1})(\hat \bm_h^{1/2} \lambda_h^{1/2})}^2 + \tau \norm{(\bI - \bLproj{1})(\hat \bm_h^{1/2} \times \bar \ptl \bm_h^1)}^2,
    \\
    \CC_2 &\coloneqq \sum\limits_{n=2}^N 
    \Big(  
    \tau \norm{(\bI - \bLproj{n}) (\hat \bm_h^n \times \bar \ptl^B \bm_h^n)}^2
    + \tau \norm{(\bI - \bLproj{n})(\hat \bm_h^n \lambda_h^n)}^2
    \\&\qq 
    + \tau^3 \norm{\bar \ptl \big((\bI - \bLproj{n}) (\hat \bm_h^n \lambda_h^n) \big)}^2
    + \tau^3 \norm{\bar \ptl \big( (\bI - \bLproj{n}) (\hat \bm_h^n \times \bar \ptl^B \bm_h^n) \big)}^2
    \Big),
\end{align*}
the \textbf{extrapolation error estimators}
\begin{align*}
    \QQ_1 &\coloneqq 
        \int_{I_1} \big(\norm{(\bm_h^1 - \hat \bm_h^1) \times \ptl_t \bM_h}^2 +  \norm{(\bm_h - \hat \bm_h) \lambda_h^1}^2\big) \dt
        + \tau^3 \norm{(\bar \ptl \bm_h^1 - \bar \ptl \hat \bm_h^1) \times \bar \ptl \bm_h^1}^2
        \\&\qq
        + \tau^2 \norm{(\bm_h^1 - \hat \bm_h^1) \bar \ptl \lambda_h^1}^2
    \\
    \QQ_2 &\coloneqq 
        \sum\limits_{n=2}^N \Big( \int_{I_n} \big(\norm{(\bm_h^n - \hat \bm_h^n) \times \ptl_t \bM_h}^2 + \norm{(\bm_h - \hat \bm_h) \lambda_h^n}^2 \big) \dt
        \\&\qq+ \tau^3 \norm{(\bar \ptl \bm_h^n - \bar \ptl \hat \bm_h^n) \times \bar \ptl^B \bm_h^n}^2
        + \tau^3 \norm{(\bm_h^n - \hat \bm_h^n)\bar \ptl \lambda_h^n}^2 \Big),
\end{align*}
and the \textbf{initial error estimators}
\begin{equation}
\begin{aligned}
    \II_1 &\coloneqq \tau^3 \norm{\bLproj{1}(\hat \bm_h^1 \times \bar \ptl^2 \bm_h^1)}^2 + \tau^5 \norm{\bLproj{1}(\bar \ptl \lambda_h^1 \bar \ptl \hat \bm_h^1)}^2,
    \qq &\II_2 \coloneqq \tau^3 \norm{\bPsi_0}^2,
    \\
    \mathclap{\hspace{12cm}\II_3 \coloneqq \tau^5 \norm{\bar \ptl \bm_h^1 \times \bar \ptl^2 \bm_h^2}^2
        + \tau^5 \norm{\bar \ptl \hat \bm_h^1 \bar \ptl \lambda_h^1}^2
        + \tau^5 \norm{\bar \ptl \bm_h^1 \bar \ptl \lambda_h^1}^2,}
\end{aligned}
\end{equation}
where $\bPsi_0 \coloneqq \alpha \bar \ptl^2 \bm_h^2 + \bar \ptl \hat \bm_h^1 \times \bar \ptl \bm_h^1 + \hat \bm_h^1 \times \bar \ptl^2 \bm_h^2 - \bar \ptl (\Delta_h^1 \bm_h^1) + \hat \bm_h^1 \bar \ptl \lambda_h^1 + \bar \ptl \hat \bm_h^1 \lambda_h^1 - \bar \ptl \bff_h^1$.
\end{Definition}

In contrast to standard parabolic problems such as the heat equation, the tangent-plane formulation introduces additional consistency terms originating from the nonlinear projection and the predictor.
Our main result, presented in the theorem below, establishes a computable \emph{a posteriori} error bound that separates the distinct sources of discretization error. 

\begin{theorem}[A~posteriori error estimate]
    \label{THM: main theorem}
    Let $\bm$ denote the exact solution of \eqref{EQ: linearized LLG analytical}, and let $\bm_h^n$ be the solution of \eqref{EQ: discrete LLG} for $n = 2, \dots, N$, with $\bm_h^1$ given by \eqref{EQ: init trapezoidal m1}. Then
    \begin{equation}   
    \begin{aligned}
        \label{EQ: a post full}
        \norm{\grad (\bm - \bm_h)}_{\Linf(t_0, t_N; \bLspace)}^2 
        &\lesssim 
        \FF_1 + \FF_2 
        + \tau^4 \big( \EE_1 + \EE_2 + \EE_3  + \EE_4 + \EE_5 + \EE_6 \big) 
        + \Lambda_1 + \Lambda_2 + \Lambda_3 \\
        &\qquad
        + \PP
        + \Xi_1 + \Xi_2 + \Xi_3 
        + \CC_1 + \CC_2 
        + \QQ_1 + \QQ_2
        + \II_1 + \II_2 + \II_3,
    \end{aligned}
    \end{equation}
    where the estimators are defined in Definition~\ref{DEF: error est.}.  
    The hidden constant depends only on $\alpha$, $\norm{\bM_h}_{\bWkp{1}{\infty}}$, $\norm{\ptl_t \bM_h}_{\bWkp{1}{\infty}}$, $\etaWinfty{\bm_h^n}$, and $\norm{\lambda_h}_{\Linf}$.
\end{theorem}

Before presenting the proof of Theorem~\ref{THM: main theorem}, we first introduce appropriate time-space reconstructions and derive a corresponding parabolic error equation.
Based on the elliptic reconstruction, we define the linear-in-time reconstruction
\begin{align}
    \label{EQ: elliptic temp reconstruction}
        \bw(t)
    = \bRR \bm_h^n + (t - t_n) \bar \ptl \bRR \bm_h^n,
\end{align}
for $t \in I_n$, $n \geq 1$, and the quadratic (three-point) time-space reconstruction
\begin{align}
    \label{EQ: three-point time-space reconstruction}
    \bW(t) = \bw(t) + \tfrac12 (t - t_n)(t - t_{n-1}) \bar \ptl^2 \bRR \bm_h^n,
\end{align}
for $t \in I_n$, $n \geq 2$.
For the initial interval $I_1$, we set
\begin{align*}
    \bW(t) = \bw(t) + \tfrac12 (t - t_0)(t - t_{1}) \bar \ptl^2 \bRR \bm_h^2.
\end{align*}
For the Lagrangian multiplier, we employ a linear-in-time reconstruction, namely
\begin{align}
    \lambda_h(t) = \lambda_h^n + (t-t_n) \bar \ptl \lambda_h^n,
\end{align}
for $t \in I_n$, $n \ge 1$.
The subsequent analysis requires $\bW(t), \ptl_t \bW(t) \in \bWspacekp{1}{\infty}$ a.e. in $(0, T)$.
This is a natural counterpart of the regularity assumptions imposed on the exact solution in \emph{a~priori} error analysis for the LLG equation,~cf.~\cite[(3.2)]{AFKL:2021}.
To this end, we assume that the elliptic reconstruction satisfies
\begin{align}
    \bRR \bm_h^n \in \bW^{1,\infty}(\Omega),
\end{align}
cf.~Assumption~\ref{AS: Linfty estimators}, for all $n \ge 1$.
Since the time-space reconstruction $\bW$ is obtained from the elliptic
reconstructions by piecewise polynomial interpolation in time, we have
$\bW, \ptl_t \bW\in L^\infty(0, T; \bWspacekp{1}{\infty})$.

As a first step, we derive the error equation using both the three-point reconstruction $\bM_h$ and the time-space reconstruction $\bw$. 
To this end, we introduce the error variables
\begin{align}
    \label{EQ: Error variables}
    \errM \coloneqq \bM_h - \bm,
    \qquad
    \errm \coloneqq \bm_h - \bm,
    \qquad
    \errW \coloneqq \bW - \bm,
    \qquad
    \errw \coloneqq \bw - \bm,
    \qquad
    \errlam \coloneqq \lambda_h - \lambda.
\end{align}
In the parabolic error equation below, the term $\dual{\errm \times \ptl_t \bM_h, \bphi}$ does not involve a time derivative of $\bm$, thus a linear-in-time reconstruction is sufficient. 
Employing the three-point reconstruction in this term merely introduces additional a~posteriori terms without providing any advantage.

\begin{Lemma}[Parabolic error equation]
    \label{LEMMA: parabolic err eq} 
    Let $(\bm, \lambda)$ be the exact solution of \eqref{EQ: linearized LLG analytical SPP} and $(\bm_h^n, \lambda_h^n)$ the solution of \eqref{EQ: discrete LLG SPP}.
    Then, for $t \in I_n$ and $n \ge 2$, we have
    \begin{equation}
    \label{EQ:parabolic err eq reconstruction}
    \begin{aligned}
        {\alpha} \dual{\ptl_t \errM, \bphi} + \dual{\bm \times \ptl_t \errM, \bphi} + \dual{\errm \times \ptl_t \bM_h, \bphi} + a(\errw, \bphi) + b^{\bm}(\errlam, \bphi) + b^{\errm}(\lambda_h, \bphi)
        &= \dual{\br_h, \bphi},
    \end{aligned}
    \end{equation} 
    for all $\bphi \in \Lsym^2(t_1, T; \bV)$, 
    where 
    \begin{equation}
    \label{EQ: residual}
    \begin{aligned}
        \dual{\br_h, \bphi}
        &=
        \dual{\bff_h^n - \bff, \bphi} + {\alpha} \dual{(\bar \ptl^B - \bar \ptl_n^B) \bm_h^n, \bphi} 
            + \dual{(\bm_h^n - \hat \bm_h^n) \times \ptl_t \bM_h, \bphi}
            \\&\qqq
            + (t-t_n) \dual{(\bar \ptl \bm_h^n - \bar \ptl \hat \bm_h^n) \times \bar \ptl^B \bm_h^n, \bphi}
            + (t-t_n)^2 \dual{\bar \ptl \bm_h^n \times \bar \ptl^2 \bm_h^n, \bphi} 
            \\ &\qqq
            + \dual{(\bm_h - \hat \bm_h) \lambda_h^n, \bphi}
            + (t - t_n) \dual{(\bm_h^n - \hat \bm_h^n) \bar \ptl \lambda_h^n, \bphi}
            \\&\qqq
            + (t - t_n)^2 \dual{\bar \ptl \bm_h^n \bar \ptl \lambda_h^n, \bphi}
            + (t - t_n) \dual{\bPsi, \bphi}
            \\ &\qqq
            + \dual{(\bI - \bLproj{n}) (\hat \bm_h^n \times \bar \ptl^B \bm_h^n), \bphi} + \dual{(\bI - \bLproj{n})(\hat \bm_h^n \lambda_h^n), \bphi}
    \end{aligned}
    \end{equation}
    for all $\bphi \in \Lsym^2(t_1, T; \bV)$ and with
    \begin{align}
        \label{EQ: Psi}
        \bPsi \coloneqq {\alpha} \bar \ptl^2 \bm_h^n + \bar \ptl \hat \bm_h^n \times \bar \ptl^B \bm_h^n + \hat \bm_h^n \times \bar \ptl^2 \bm_h^n - \bar \ptl (\Delta_h^n \bm_h^n) + \hat \bm_h^n \bar \ptl \lambda_h^n + \bar \ptl \hat \bm_h^n \lambda_h^n.
    \end{align}
\end{Lemma}
\begin{proof}
    We start by using the orthogonality of the $\Lsym^2$ projection $ \bLproj{n}$ onto $\bV_h^n$,
    the definition of the discrete Laplacian \eqref{EQ: discrete Laplacian def} and the discrete equation \eqref{EQ: discrete LLG SPP} to obtain 
    \begin{equation}
    \label{EQ:pointwise LLG}
    \begin{aligned}
        &\dual{{\alpha} \bar \ptl_n^B \bm_h^n + \bLproj{n} (\hat \bm_h^n \times \bar \ptl^B \bm_h^n) - \Delta_h^n \bm_h^n + \bLproj{n} (\hat \bm_h^n \lambda_h^n) - \bff_h^n, \bphi}
        \\&\qq= 
         \dual{{\alpha} \bar \ptl^B \bm_h^n + \hat \bm_h^n \times \bar \ptl^B \bm_h^n - \bff_h^n, \bLproj{n} \bphi} + a(\bm_h^n, \bLproj{n} \bphi) + b^{\hat{\bm}_h^n}(\lambda_h^n, \bLproj{n} \bphi)
        = 0
    \end{aligned}
    \end{equation}
    for $n \ge 2$ and for all $\bphi \in \Lsym^2(t_1, T; \bV)$.
    Using the the weak formulation \eqref{EQ: linearized LLG analytical SPP} and subtracting equation \eqref{EQ:pointwise LLG}, we have
    \begin{equation}
    \label{EQ: elliptic err eq 1.1}
    \begin{aligned}
        &{\alpha} \dual{\ptl_t \errM, \bphi} + \dual{\bm \times \ptl_t \errM, \bphi} + \dual{\errm \times \ptl_t \bM_h, \bphi} + a(\errw, \bphi) + b^{\bm}(\errlam, \bphi) + b^{\errm}(\lambda_h, \bphi)
        \\& \qq=
            \dual{\bff_h^n - \bff, \bphi} 
            + \alpha \dual{\ptl_t \bM_h - \bar \ptl_n^B \bm_h^n, \bphi} + \dual{\bm_h \times \ptl_t \bM_h - \bLproj{n} (\hat \bm_h^n \times \bar \ptl^B \bm_h^n), \bphi} 
            \\ &\qqq
            + a(\bw, \bphi) + \dual{\Delta_h^n \bm_h^n, \bphi}
            + \dual{\bm_h \lambda_h - \bLproj{n} (\hat \bm_h^n \lambda_h^n), \bphi}.
    \end{aligned}
    \end{equation}
    Further, we have, by the definition of $\bw$ \eqref{EQ: Error variables}, as well as the definition of the elliptic reconstruction \eqref{EQ: elliptic reconstruction}
    \begin{equation}\label{EQ: a(bw,phi) = ...}
    \begin{aligned}
        a(\bw, \bphi) 
        &= a(\bRR \bm_h^n, \bphi) + (t - t_n) \, a(\bar \ptl \bRR \bm_h^n, \bphi)
        \\&= - \dual{\Delta_h^n \bm_h^n, \bphi} - (t - t_n) \dual{ \bar \ptl (\Delta_h^n \bm_h^n), \bphi}.
    \end{aligned}   
    \end{equation}
    Rewriting the second on the right-hand side of \eqref{EQ: elliptic err eq 1.1} with the definition of the three-point reconstruction \eqref{EQ:def-Mh_uni},
    adding and subtracting $\dual{\hat \bm_h^n \times \bar \ptl^B \bm_h^n, \bphi}$ to the third term and $\dual{\hat \bm_h^n \lambda_h^n, \bphi}$ to the last term on the right-hand side, we obtain
    \begin{equation}
    \label{EQ: elliptic err eq 2.0}
    \begin{aligned}
        &{\alpha} \dual{\ptl_t \errM, \bphi} + \dual{\bm \times \ptl_t \errM, \bphi} + \dual{\errm \times \ptl_t \bM_h, \bphi} + a(\errw, \bphi) + b^{\bm}(\errlam, \bphi) + b^{\errm}(\lambda_h, \bphi)
        \\& \qq=
            \dual{\bff_h^n - \bff, \bphi} + \alpha \dual{(\bar \ptl^B - \bar \ptl_n^B) \bm_h^n, \bphi} + \dual{\bm_h \times \ptl_t \bM_h - \hat \bm_h^n \times \bar \ptl^B \bm_h^n, \bphi} 
            \\ &\qqq
            + \dual{\bm_h \lambda_h - \hat \bm_h^n \lambda_h^n, \bphi}
            + (t - t_n) \Big( {\alpha} \dual{\bar \ptl^2 \bm_h^n, \bphi} - \dual{\bar \ptl (\Delta_h^n \bm_h^n), \bphi} \Big)
            \\ &\qqq
            + \dual{(\bI - \bLproj{n}) (\hat \bm_h^n \times \bar \ptl^B \bm_h^n), \bphi} + \dual{(\bI - \bLproj{n})(\hat \bm_h^n \lambda_h^n), \bphi}.
    \end{aligned}
    \end{equation}
    We continue by rewriting the third term on the right-hand side of \eqref{EQ: elliptic err eq 2.0} by 
    \begin{equation}
    \label{EQ: ell err eq rewrite 2.1}
    \begin{aligned}
        &\dual{\bm_h \times \ptl_t \bM_h - \hat \bm_h^n \times \bar \ptl^B \bm_h^n, \bphi}
        \\&\qq= 
        \dual{(\bm_h^n - \hat \bm_h^n) \times \ptl_t \bM_h, \bphi}
        + (t-t_n) \dual{(\bar \ptl \bm_h^n - \bar \ptl \hat \bm_h^n) \times \bar \ptl^B \bm_h^n, \bphi}
        \\&\qqq
        + (t-t_n)^2 \dual{\bar \ptl \bm_h^n \times \bar \ptl^2 \bm_h^n, \bphi} 
        + (t - t_n) \dual{\bar \ptl \hat \bm_h^n \times \bar \ptl^B \bm_h^n + \hat \bm_h^n \times \bar \ptl^2 \bm_h^n, \bphi}.
    \end{aligned}
    \end{equation}
    Similarly, 
    we rewrite the fourth term on the right-hand side of \eqref{EQ: elliptic err eq 2.0} by
    \begin{equation}
    \label{EQ: ell err eq rewrite 2.2}
    \begin{aligned}
        \dual{\bm_h \lambda_h - \hat \bm_h^n \lambda_h^n, \bphi}
        &= 
        \dual{(\bm_h - \hat \bm_h) \lambda_h^n, \bphi}
        + (t - t_n)\dual{(\bm_h^n - \hat \bm_h^n) \bar \ptl \lambda_h^n, \bphi}
        \\&\qq
        + (t - t_n) \dual{\hat \bm_h^n \bar \ptl \lambda_h^n + \bar \ptl \hat \bm_h^n \lambda_h^n, \bphi}
        + (t - t_n)^2 \dual{\bar \ptl \bm_h^n \bar \ptl \lambda_h^n, \bphi}.
    \end{aligned}
    \end{equation}
    Finally, by inserting \eqref{EQ: ell err eq rewrite 2.1} and \eqref{EQ: ell err eq rewrite 2.2} into the right-hand side of \eqref{EQ: elliptic err eq 2.0}, we conclude the proof.
\end{proof}

Having established the parabolic error equation, we can now proceed with the proof of Theorem~\ref{THM: main theorem}. For clarity and to emphasize the key ideas, we postpone the proofs of the a~posteriori terms and the initial time-step a~posteriori bound to the next section.
\begin{proof}[Proof of Theorem~\ref{THM: main theorem}]
Rewriting the parabolic error equation \eqref{EQ:parabolic err eq reconstruction}, we have 
\begin{align*}
    &{\alpha}\dual{\ptl_t \errW, \bphi} + \dual{\bm \times \ptl_t \errW, \bphi} + \dual{\errW \times \ptl_t \bM_h, \bphi} + a(\errw, \bphi) + b^{\bm}(\errlam, \bphi) + b^{\errW}(\lambda_h, \bphi)
    \\ 
    &\qqq = {\alpha} \dual{\ptl_t (\bW - \bM_h), \bphi} + \dual{\bm \times \ptl_t (\bW - \bM_h), \bphi} + \dual{(\bW - \bm_h) \times \ptl_t \bM_h, \bphi} 
    \\&\qqq \qq+ \dual{(\bW - \bm_h)\lambda_h, \bphi} + \dual{\br_h, \bphi}.
\end{align*}
Now, since $\norm{\bm}_{\bLinf} = 1$ and $\dual{\bm \times \ptl_t \errW, \ptl_t \errW} = 0$, testing with $\bphi = \ptl_t \errW$ yields
\begin{align*}
        &{\alpha} \norm{\ptl_t \errW}^2 + a(\errw, \ptl_t \errW) + b^{\bm}(\errlam, \ptl_t \errW) + b^{\errW}(\lambda_h, \ptl_t \errW)
    \\ &\qq \le
        \norm{\ptl_t \errW} \Big(\norm{\ptl_t \bM_h}_{\bLinf} \norm{\errW} + {(1 + \alpha)} \norm{\ptl_t (\bW - \bM_h)} 
       \\&\qqq+ \norm{(\bW - \bm_h) \times \ptl_t \bM_h} + \norm{(\bW - \bm_h)\lambda_h} + \norm{\br_h}\Big).
\end{align*}
Next, we treat the bilinear form $a$ as
\begin{align*}
    a(\errw, \ptl_t \errW)
    &=
    a(\errW, \ptl_t \errW) + a(\bw - \bW, \ptl_t \errW)
    =
    a(\errW, \ptl_t \errW) + \dual{ -\Delta_h (\bw - \bW), \ptl_t \errW},
\end{align*}
where $- \Delta_h (\bw - \bW) = \frac12 (t-t_n)(t - t_{n-1}) \bar\ptl^2 (-\Delta_h^n) \bm_h^n$ by the definition \eqref{EQ: elliptic temp reconstruction} and \eqref{EQ: three-point time-space reconstruction} of the reconstructions for $t \in I_n$,
since $\bw$ and $\bW$ are defined through pointwise (in time) elliptic reconstructions.
Further, using $a(\errW, \ptl_t \errW) = \frac12 \frac{\de}{\dt} \norm{\grad \errW}^2,$ we obtain
\begin{equation}
\label{EQ: Proof Main THM Eq1}
\begin{aligned}        
    &{\alpha}\norm{\ptl_t \errW}^2 + \frac12 \frac{\de}{\dt} \norm{\grad \errW}^2 + b^{\bm}(\errlam, \ptl_t \errW) + b^{\errW}(\lambda_h, \ptl_t \errW)
    \\ &\qq \le
        \norm{\ptl_t \errW} \Big(\norm{\ptl_t \bM_h}_{\bLinf} \norm{\errW} + {(1 + \alpha)} \norm{\ptl_t (\bW - \bM_h)} 
        + \norm{(\bW - \bm_h) \times \ptl_t \bM_h}
        \\&\qqq
         + \norm{(\bW - \bm_h)\lambda_h} + \norm{\br_h} + \norm{\Delta_h (\bw - \bW)}\Big).
\end{aligned}
\end{equation}
To reduce the notation we define $c_\infty \coloneqq \norm{\ptl_t \bM_h}_{\bLinf} + \norm{\lambda_h}_{\Linf}$, as well as,
\begin{equation}
\begin{aligned}
    \label{EQ: def of AA}
    \AA &\coloneqq 
        {(1 + \alpha)} \norm{\ptl_t (\bW - \bM_h)} 
        + c_\infty \norm{(\bW - \bm_h)} 
        + \norm{\br_h} 
        + \norm{\Delta_h (\bw - \bW)},
\end{aligned}
\end{equation}
Thus we can rewrite equation \eqref{EQ: Proof Main THM Eq1} by
\begin{align*}
        &{\alpha}\norm{\ptl_t \errW}^2 + \frac12 \frac{\de}{\dt} \norm{\grad \errW}^2 + b^{\bm}(\errlam, \ptl_t \errW) + b^{\errW}(\lambda_h, \ptl_t \errW)
    \le
        \norm{\ptl_t \errW} \Big(\norm{\ptl_t \bM_h}_{\bLinf} \norm{\errW} + \AA \Big).
\end{align*}
Further, using the parabolic error equation \eqref{EQ:parabolic err eq reconstruction}, $\bm \cdot \ptl_t \bm = 0$ a.e. and $\bm \cdot \bP(\bm) \ptl_t \bM_h = 0$ a.e., we obtain
\begin{align*}
    b^{\bm}(\errlam, \ptl_t \errW)
    &= b^{\bm}(\errlam, \ptl_t \bW)
    = b^{\bm}(\errlam, (\bI - \bP(\bm)) \ptl_t \bW)
    \\&\le 
    \norm{(\bI - \bP(\bm)) \ptl_t \bW} \Big({(1 + \alpha)} \norm{\ptl_t \errM} + c_\infty \norm{\errm} + \norm{\br_h} + \norm{\Delta_h(\bw - \bW)}\Big)
    \\&\qq + \norm{\grad \errW} \norm{\grad(\bI - \bP(\bm)) \ptl_t \bW},
\end{align*}
since
\begin{align*}
    a(\errw, (\bI - \bP(\bm)) \ptl_t \bW) = a(\errW, (\bI - \bP(\bm)) \ptl_t \bW) + \dual{- \Delta_h (\bw - \bW), (\bI - \bP(\bm)) \ptl_t \bW}.
\end{align*}
Using $\norm{\errm} \le \norm{\errW} + \norm{\bW - \bm_h}$, $\norm{\ptl_t \errM} \le \norm{\ptl_t \errW} + \norm{\ptl_t (\bW - \bM_h)}$ and the definition of $\AA$ \eqref{EQ: def of AA} yields
\begin{equation}
\begin{aligned}
    \label{EQ: est b^m}
    b^{\bm}(\errlam, \ptl_t \errW)
    &\le 
    \norm{(\bI - \bP(\bm)) \ptl_t \bW} \Big({(1 + \alpha)} \norm{\ptl_t \errW} + c_\infty \norm{\errW} + \AA \Big) 
    \\&\qq 
    + \norm{\grad \errW} \norm{\grad(\bI - \bP(\bm)) \ptl_t \bW}.
\end{aligned}
\end{equation}
Utilizing $b^{\errW}(\lambda_h, \ptl_t \errW) \le \norm{\lambda_h}_{\Linf} \norm{\errW} \norm{\ptl_t \errW}$ and \eqref{EQ: est b^m}, we deduce
\begin{align*}
        {\alpha} \norm{\ptl_t \errW}^2 + \frac12 \frac{\de}{\dt} \norm{\grad \errW}^2
        &\le
        \norm{\ptl_t \errW} \Big(c_\infty \norm{\errW} + \AA\Big) 
        + \norm{\grad \errW} \norm{\grad(\bI - \bP(\bm)) \ptl_t \bW}
        \\ &\qq
        + \norm{(\bI - \bP(\bm)) \ptl_t \bW} \Big({(1 + \alpha)} \norm{\ptl_t \errW} + c_\infty \norm{\errW} + \AA \Big) .
\end{align*}
Applying Young's inequality, making use of
\begin{align*}
    \norm{\ptl_t \errW}^2 \ge \frac{\de}{\dt} \norm{\errW}^2 - \norm{\errW}^2
\end{align*}
and subtracting $\frac12 \alpha \norm{\ptl_t \errW}^2$ from both sides, we obtain
\begin{align*}
        {\min\{\alpha, 1 \}}\frac{\de}{\dt} \norm{\errW}_{\bHk{1}}^2 
        &\le
        \norm{\errW}^2 
        + {\frac2\alpha} \Big(c_\infty \norm{\errW} + \AA\Big)^2
        + 2 \norm{\grad \errW} \norm{\grad(\bI - \bP(\bm)) \ptl_t \bW}
        \\ &\quad
        + {\frac{2(1+\alpha)^2}{\alpha}} \norm{(\bI - \bP(\bm)) \ptl_t \bW}^2 
        + 2 \norm{(\bI - \bP(\bm)) \ptl_t \bW} \Big(c_\infty \norm{\errW} + \AA \Big).
\end{align*}
Define 
\begin{align}
    \label{EQ: def gammas}
    \gamma_{0, \bW} \coloneqq (1 + \norm{\bW}_{\bLinf}) \norm{\ptl_t \bW}_{\bLinf}, \qq \gamma_{1, \bW} \coloneqq \gamma(\bW, \ptl_t \bW),
\end{align}
where $\gamma$ is defined in Lemma~\ref{LEM: Lipschitz bound}.
Next, we estimate using the Lipschitz-type bound of the orthogonal projection in Lemma~\ref{LEM: Lipschitz bound}
\begin{align*}
    \norm{(\bI - \bP(\bm)) \ptl_t \bW}
    &\le 
    \norm{(\bI - \bP(\bW)) \ptl_t \bW}
    + \gamma_{0, \bW} \norm{\errW},
    \\ 
    \norm{\grad(\bI - \bP(\bm)) \ptl_t \bW}
    &\le 
    \norm{\grad(\bI - \bP(\bW)) \ptl_t \bW} + \gamma_{1, \bW} \norm{\errW}_{\bHk{1}}.
\end{align*}
Applying the above equations yields
\begin{align*}
        {\min\{\alpha, 1 \}}\frac{\de}{\dt} \norm{\errW}_{\bHk{1}}^2  
    &\le
        \norm{\errW}^2 + {\frac2\alpha} \Big(c_\infty \norm{\errW} + \AA \Big)^2 + 2 \norm{\grad \errW} \norm{\grad(\bI - \bP(\bW)) \ptl_t \bW} 
        \\ &\quad 
        + {\frac{2(1+\alpha)^2}{\alpha}} \big( \norm{(\bI - \bP(\bW)) \ptl_t \bW} + \gamma_{0, \bW} \norm{\errW} \big)^2
        \\ &\quad 
        + 2 \big(\norm{(\bI - \bP(\bW)) \ptl_t \bW} + \gamma_{0, \bW} \norm{\errW}\big) \Big(c_\infty \norm{\errW} + \AA \Big)
        \\&\quad
        + 2 \norm{\grad \errW} \gamma_{1, \bW} \norm{\errW}_{\bHk{1}}.
\end{align*}
Exploiting Young's inequality, $(a+b)^2 \le 2 (a^2 + b^2)$ and reordering yields
\begin{align*}
        {\min\{\alpha, 1 \}} \frac{\de}{\dt} \norm{\errW}_{\bHk{1}}^2  
    &\le
        c_0 \norm{\errW}^2 + c_1 \norm{\errW}_{\bHk{1}}^2 
        + \Big({\frac{4(1+\alpha)^2}{\alpha}} + 1 + c_\infty^2\Big) \norm{(\bI - \bP(\bW)) \ptl_t \bW}^2 
        \\ &\quad 
        + {\Big(\frac4\alpha + 2\Big)} \AA^2
        + \norm{\grad(\bI - \bP(\bW)) \ptl_t \bW}^2,
\end{align*}
where $c_0 \coloneqq 2 + {\frac4\alpha} c_\infty^2 + 2 c_\infty \gamma_{0, \bW} + {(\frac{4(1+\alpha)^2}{\alpha} + 1)} \gamma_{0, \bW}^2$ and $c_1 \coloneqq 1 + 2 \gamma_{1,\bW}$.
Integrating from $t = t_1$ to $t_N$, we obtain
\begin{align*}
        \norm{\errW (t_N)}_{\bHk{1}}^2  
    &\le
        \norm{\errW (t_1)}_{\bHk{1}}^2  
        +
        \frac{1}{{\min\{\alpha, 1 \}}}
        \int_{t_1}^{t_N}
        \Big(c_0 \norm{\errW}^2 + c_1 \norm{\errW}_{\bHk{1}}^2\Big) \dt
        + \frac{1}{{\min\{\alpha, 1 \}}} \tilde \AA,
\end{align*}
where
\begin{align*}
    \tilde \AA &\coloneqq \int_{t_1}^{t_N} \bigg\{
        \tilde c_\infty \norm{(\bI - \bP(\bW)) \ptl_t \bW}^2 
        + {\Big(\frac4\alpha + 2\Big)} \AA^2
        + \norm{\grad(\bI - \bP(\bW)) \ptl_t \bW}^2 \bigg\}\dt
\end{align*}
with $\tilde c_\infty \coloneqq {\frac{4(1+\alpha)^2}{\alpha}} + 1 + c_\infty^2$.
Thus, estimating $\norm{\errW}^2 \le \norm{\errW}_{\bHk{1}}^2$, we conclude by Gronwall's inequality
\begin{align*}
        \norm{\errW (t_N)}_{\bHk{1}}^2  
    &\le
        \Big(\norm{\errW (t_1)}_{\bHk{1}}^2 + \frac1{{\min\{\alpha, 1 \}}}\tilde \AA\Big)
        \ee^{
        \int_{t_1}^{t_N}
         \frac{c_0 + c_1}{{\min\{\alpha, 1 \}}}\dt}.
\end{align*}
We finalize the a~posteriori estimate \eqref{EQ: a post full} for $t \in [t_1, t_N]$ by 
\begin{align*}
    \norm{\grad (\bm - \bm_h)(t)}^2
    \le
    \norm{\grad (\bM_h - \bm_h)(t)}^2 
    + \norm{\grad (\bW - \bM_h)(t)}^2
    + \norm{\grad \errW(t)}^2,
\end{align*}
where we estimate the remaining terms in the next section.
\end{proof}

\subsection{Remaining error terms}
In this section, we provide the remaining a posteriori error bounds to conclude     Theorem~\ref{THM: main theorem}.
The error estimates are mostly derived by adapting the approach of \cite{BaenschBrenner:2019}, where comparable error terms were analyzed.
We first estimate the initial time-step, which is provided by \eqref{EQ: init trapezoidal m1}. 
To simplify the notation, we assume $\bV_h^0 = \bV_h^1 = \bV_h^2$.

\begin{Lemma}[Initial time-step a posteriori estimate] 
    Assume $V_h^0 = V_h^1 = V_h^2$. Let $(\bm, \lambda)$ be the exact solution of \eqref{EQ: linearized LLG analytical SPP}. Further, let $(\bm_h^1, \lambda_h^1)$ be the solution of \eqref{EQ: init trapezoidal m1} and $(\bm_h^2, \lambda_h^2)$ be the solution of \eqref{EQ: discrete LLG SPP}. Then 
    \begin{align*}
        &\norm{\grad (\bm - \bm_h)}_{\Linf(0, t_1; \bLspace)}^2  
        \\&\quad \lesssim 
        \norm{\grad( \bP_0^0 \bm(0) - \bm(0))}^2
        + \norm{\grad (\bM_h - \bm_h)}_{\Linf(0, t_1; \bLspace)}^2
        + \norm{\grad(\bW - \bM_h)}_{\Linf(0, t_1; \bLspace)}^2
        \\&\qq
        + \int_{0}^{t_1} \Big( 
        \norm{\ptl_t (\bW - \bM_h)}^2 
        + \norm{\bW - \bm_h}^2
        + \norm{\br_h^0}^2  
        + \norm{\Delta_h (\bw - \bW)}^2
        \\ &\qqq 
        + \norm{(\bI - \bP(\bW)) \ptl_t \bW}_{\bHk{1}}^2 \Big)\dt,
    \end{align*}
    where the initial reconstruction $\bM_h$ is given in \eqref{EQ: initial tp-reconstruction} and $\br_h^0$ is defined by
    \begin{equation}        
    \label{EQ: r_h^i}
    \begin{split}\raisetag{-13ex}
        \dual{\br_h^0, \bphi} &=
            \dual{\bff_h - \bff, \bphi} + \dual{(\bm_h^1 - \hat \bm_h^1) \times \ptl_t \bM_h, \bphi}
            + (t-t_1) \dual{(\bar \ptl \bm_h^1 - \bar \ptl \hat \bm_h^1) \times \bar \ptl \bm_h^1, \bphi}
        \\&\qq\hspace{-8pt}
            + (t-t_1)(t-t_{1/2}) \dual{\bar \ptl \bm_h^1 \times \bar \ptl^2 \bm_h^2, \bphi} 
            + \dual{(\bm_h - \hat \bm_h) \lambda_h^1, \bphi}
        \\&\qq\hspace{-8pt}
            + (t - t_1) \dual{(\bm_h^1 - \hat \bm_h^1) \bar \ptl \lambda_h^1, \bphi}
            - \frac{\tau^2}{4}\dual{\bar \ptl \hat \bm_h^1 \bar \ptl \lambda_h^1, \bphi}
            + (t - t_1)^2 \dual{\bar \ptl \bm_h^1 \bar \ptl \lambda_h^1, \bphi}
        \\&\qq\hspace{-8pt}
            + \dual{(\bI - \bP_0^1) (\hat \bm_h^{1/2} \times \bar \ptl \bm_h^1), \bphi} + \dual{(\bI - \bP_0^1)(\hat \bm_h^{1/2} \lambda_h^{1/2}), \bphi}
            + (t - t_{1/2}) \dual{\bPsi_0, \bphi}
    \end{split}
    \end{equation}
    with $\bPsi_0 = \alpha \bar \ptl^2 \bm_h^2 + \bar \ptl \hat \bm_h^1 \times \bar \ptl \bm_h^1 + \hat \bm_h^1 \times \bar \ptl^2 \bm_h^2 - \bar \ptl (\Delta_h^1 \bm_h^1) + \hat \bm_h^1 \bar \ptl \lambda_h^1 + \bar \ptl \hat \bm_h^1 \lambda_h^1 - \bar \ptl \bff_h^1$.
\end{Lemma}
\begin{proof}
    Since $\bV_h^1 = \bV_h^0$, we have $\bar \ptl \bm_h^1 = \bP_0^1 \bar \ptl \bm_h^1 = \bar \ptl_1 \bm_h^1$. 
    Similarly to \eqref{EQ:pointwise LLG}, using the orthogonality of the $\Lsym^2$ projection $ \bLproj{1}$ onto $\bV_h^1 = \bV_h^0$,
    the definition of the discrete Laplacian \eqref{EQ: discrete Laplacian def} and the weak formulation of the trapezoidal scheme \eqref{EQ: init trapezoidal m1} we obtain 
    \begin{equation}
    \label{EQ:pointwise LLG I_0}
    \begin{aligned}
        &\dual{\alpha\bar \ptl \bm_h^1 + \bP_0^1(\hat \bm_h^{1/2} \times \bar \ptl \bm_h^1) - \Delta_h^1 \bm_h^{1/2} + \bP_0^1(\hat \bm_h^{1/2} \lambda_h^{1/2}) - \bff_h^{1/2}, \bphi}
        \\&\qq= 
        \dual{\alpha\bar \ptl \bm_h^1 + \hat \bm_h^{1/2} \times \bar \ptl \bm_h^1 - \Delta_h^1 \bm_h^{1/2} + \hat \bm_h^{1/2} \lambda_h^{1/2} - \bff_h^{1/2}, \bLproj{1} \bphi} 
        \\&\qq= 
         \dual{\alpha\bar \ptl \bm_h^1 + \hat \bm_h^{1/2} \times \bar \ptl \bm_h^1 - \bff_h^{1/2}, \bLproj{1} \bphi} + a(\bm_h^{1/2}, \bLproj{1} \bphi) + b^{\hat{\bm}_h^{1/2}}(\lambda_h^{1/2}, \bLproj{1} \bphi)
        = 0
    \end{aligned}
    \end{equation}
    for all $\bphi \in \Lsym^2(0, t_1; \bV)$.
    Thus, by applying the weak formulation \eqref{EQ: linearized LLG analytical SPP} and the definition of the error variables \eqref{EQ: Error variables}, we have for $\bphi \in \Lsym^2(0, t_1; \bV)$
    \begin{align*}
            &{\alpha}\dual{\ptl_t \errM, \bphi} + \dual{\bm \times \ptl_t \errM, \bphi} + \dual{\errm \times \ptl_t \bM_h, \bphi} + a(\errw, \bphi) + b^{\bm}(\errlam, \bphi) + b^{\errm}(\lambda_h, \bphi)
        \\ & \qq=
            - \dual{\bff, \bphi} + {\alpha} \dual{\ptl_t \bM_h, \bphi} + \dual{\bm_h \times \ptl_t \bM_h, \bphi} + a(\bw, \bphi) + b^{\bm_h}(\lambda_h, \bphi).
    \end{align*}
    Subtracting equation \eqref{EQ:pointwise LLG I_0} and adding $\dual{\hat \bm_h^{1/2} \times \bar \ptl \bm_h^1 - \hat \bm_h^{1/2} \times \bar \ptl \bm_h^1, \bphi} = 0$, as well as $\dual{\hat \bm_h^{1/2} \lambda_h^{1/2} - \hat \bm_h^{1/2} \lambda_h^{1/2}, \bphi} = 0$, yields 
    \begin{equation}
    \label{EQ: par err eq init 1}
    \begin{aligned}
        &{\alpha}\dual{\ptl_t \errM, \bphi} + \dual{\bm \times \ptl_t \errM, \bphi} + \dual{\errm \times \ptl_t \bM_h, \bphi} + a(\errw, \bphi) + b^{\bm}(\errlam, \bphi) + b^{\errm}(\lambda_h, \bphi)
        \\& \qq=
            \dual{\bff_h^{1/2} - \bff, \bphi} 
            + \alpha \dual{\ptl_t \bM_h - \bar \ptl \bm_h^1, \bphi}
            + \dual{\bm_h \times \ptl_t \bM_h - \hat \bm_h^{1/2} \times \bar \ptl \bm_h^1, \bphi} 
            \\ &\qqq
            a(\bw, \bphi) + \dual{ \Delta_h^1 \bm_h^{1/2}, \bphi}
            + \dual{\bm_h \lambda_h - \hat \bm_h^{1/2} \lambda_h^{1/2}, \bphi}
            \\ &\qqq
            + \dual{(\bI - \bP_0^1) (\hat \bm_h^{1/2} \times \bar \ptl \bm_h^1), \bphi} + \dual{(\bI - \bP_0^1)(\hat \bm_h^{1/2} \lambda_h^{1/2}), \bphi}.
    \end{aligned}
    \end{equation}
    For the second term on the right-hand side, we use the definition of the three-point reconstruction $\bM_h$ on $I_1$ \eqref{EQ: initial tp-reconstruction} to obtain
    \begin{align}
        \label{EQ: ell err eq rewrite init 1.1}
        \alpha \dual{\ptl_t \bM_h - \bar \ptl \bm_h^1, \bphi}
        = \alpha (t - t_{1/2}) \dual{\bar \ptl^2 \bm_h^2, \bphi}.
    \end{align}
    Further, we have, by the definition of $\bw$ \eqref{EQ: elliptic temp reconstruction}, as well as the definition of the elliptic reconstruction \eqref{EQ: elliptic reconstruction}
    \begin{equation}\label{EQ: a(bw,phi) = ... init}
    \begin{aligned}
        a(\bw, \bphi) 
        &= a(\bRR \bm_h^1, \bphi) + (t - t_1) \, a(\bar \ptl \bRR \bm_h^1, \bphi)
        \\&= - \dual{\Delta_h^1 \bm_h^1, \bphi} - (t - t_1) \dual{ \bar \ptl (\Delta_h^1 \bm_h^1), \bphi}.
    \end{aligned}   
    \end{equation}
    Substituting \eqref{EQ: ell err eq rewrite init 1.1} and \eqref{EQ: a(bw,phi) = ... init} into the right-hand side of \eqref{EQ: par err eq init 1} leads to
    \begin{equation}
    \label{EQ: par err eq init 2}
    \begin{aligned}
        &\alpha\dual{\ptl_t \errM, \bphi} + \dual{\bm \times \ptl_t \errM, \bphi} + \dual{\errm \times \ptl_t \bM_h, \bphi} + a(\errw, \bphi) + b^{\bm}(\errlam, \bphi) + b^{\errm}(\lambda_h, \bphi)
        \\& \qq=
            \dual{\bff_h^{1/2} - \bff, \bphi} + \dual{\bm_h \times \ptl_t \bM_h - \hat \bm_h^{1/2} \times \bar \ptl \bm_h^1, \bphi} 
            + \dual{\Delta_h^1 \bm_h^{1/2} - \Delta_h^1 \bm_h^1, \bphi}
            \\ &\qqq
            + \dual{\bm_h \lambda_h - \hat \bm_h^{1/2} \lambda_h^{1/2}, \bphi}
            + \alpha(t - t_{1/2}) \dual{\bar \ptl^2 \bm_h^2, \bphi} - (t - t_1) \dual{\bar \ptl (\Delta_h^1 \bm_h^1), \bphi}
            \\ &\qqq
            + \dual{(\bI - \bP_0^1) (\hat \bm_h^{1/2} \times \bar \ptl \bm_h^1), \bphi} + \dual{(\bI - \bP_0^1)(\hat \bm_h^{1/2} \lambda_h^{1/2}), \bphi}.
    \end{aligned}
    \end{equation}
    
    We continue with rewriting the right-hand side of \eqref{EQ: par err eq init 2}.
    Since $t - t_{1/2} = \frac{\tau}2 + t-t_1$, we have 
    \begin{align}
        \label{EQ: halfstep rewrite}
        \bff_h^1 - \bff_h^{1/2} = \frac{\tau}{2} \bar \ptl \bff_h^1 = (t-t_{1/2}) \bar \ptl \bff_h^1 - (t-t_1) \bar \ptl \bff_h^1
    \end{align}
    The same identity holds for $\hat \bm_h^1$, $\lambda_h^1$ and $\Delta_h^1 \bm_h^1$.
    Inserting \eqref{EQ: halfstep rewrite} for $\bff_h^1$ and $\Delta_h^1 \bm_h^1$ into the right-hand side of \eqref{EQ: par err eq init 2} yields
    \begin{equation}
    \label{EQ: par err eq init 3}
    \begin{aligned}
        &{\alpha}\dual{\ptl_t \errM, \bphi} + \dual{\bm \times \ptl_t \errM, \bphi} + \dual{\errm \times \ptl_t \bM_h, \bphi} + a(\errw, \bphi) + b^{\bm}(\errlam, \bphi) + b^{\errm}(\lambda_h, \bphi)
        \\& \qq=
            \dual{\bff_h - \bff, \bphi} + \dual{\bm_h \times \ptl_t \bM_h - \hat \bm_h^{1/2} \times \bar \ptl \bm_h^1, \bphi} 
            \\ &\qqq
            + \dual{\bm_h \lambda_h - \hat \bm_h^{1/2} \lambda_h^{1/2}, \bphi}
            + (t - t_{1/2}) \dual{\alpha \bar \ptl^2 \bm_h^2 - \bar \ptl (\Delta_h^1 \bm_h^1) - \bar \ptl \bff_h^1, \bphi} 
            \\ &\qqq
            + \dual{(\bI - \bP_0^1) (\hat \bm_h^{1/2} \times \bar \ptl \bm_h^1), \bphi} + \dual{(\bI - \bP_0^1)(\hat \bm_h^{1/2} \lambda_h^{1/2}), \bphi}.
    \end{aligned}
    \end{equation}
    By adding and subtracting $\dual{\hat \bm_h^1 \times \ptl_t \bM_h, \bphi}$ to the second term on the right-hand side of \eqref{EQ: par err eq init 3}, the definition of the three-point reconstruction $\bM_h$ on $I_1$ \eqref{EQ: initial tp-reconstruction} and \eqref{EQ: halfstep rewrite}, we obtain
    \begin{equation}
    \label{EQ: par err eq init rewrite 3.1}
    \begin{aligned}
        &\dual{\bm_h \times \ptl_t \bM_h - \hat \bm_h^{1/2} \times \bar \ptl \bm_h^1, \bphi}
        \\&\qq= 
        \dual{(\bm_h^1 - \hat \bm_h^1) \times \ptl_t \bM_h, \bphi} + (t - t_1) \dual{\bar \ptl \bm_h^1 \times \ptl_t \bM_h, \bphi}
        + \dual{(\hat \bm_h^1 - \hat \bm_h^{1/2}) \times \bar \ptl \bm_h^1, \bphi} 
        \\&\qqq
        + (t-t_{1/2}) \dual{\hat \bm_h^1 \times \bar \ptl^2 \bm_h^2, \bphi} 
        \\&\qq= 
        \dual{(\bm_h^1 - \hat \bm_h^1) \times \ptl_t \bM_h, \bphi}
        + (t-t_1) \dual{(\bar \ptl \bm_h^1 - \bar \ptl \hat \bm_h^1) \times \bar \ptl \bm_h^1, \bphi}
        \\&\qqq
        + (t-t_1)(t-t_{1/2}) \dual{\bar \ptl \bm_h^1 \times \bar \ptl^2 \bm_h^2, \bphi} 
        + (t - t_{1/2}) \dual{\bar \ptl \hat \bm_h^1 \times \bar \ptl \bm_h^1 + \hat \bm_h^1 \times \bar \ptl^2 \bm_h^2, \bphi}.
    \end{aligned}
    \end{equation}
    Similarly, adding and subtracting first $\dual{\hat \bm_h \lambda_h^1, \bphi}$ and second $(t - t_1) \dual{\hat \bm_h^1 \bar \ptl \lambda_h^1, \bphi}$ to the third term of the right-hand side \eqref{EQ: par err eq init 3} yields
    \begin{align*}
        &\dual{\bm_h \lambda_h - \hat \bm_h^{1/2} \lambda_h^{1/2}, \bphi}
        \\ &\qqq= 
        \dual{(\bm_h - \hat \bm_h) \lambda_h^1, \bphi}
        + (t - t_1) \dual{\bm_h \bar \ptl \lambda_h^1, \bphi}
        + \dual{\hat \bm_h \lambda_h^1 - \hat \bm_h^{1/2} \lambda_h^{1/2}, \bphi}
        \\&\qqq= 
        \dual{(\bm_h - \hat \bm_h) \lambda_h^1, \bphi}
        + (t - t_1) \dual{(\bm_h^1 - \hat \bm_h^1) \bar \ptl \lambda_h^1, \bphi}
        + (t - t_1)^2\dual{\bar \ptl \bm_h^1 \bar \ptl \lambda_h^1, \bphi}
        \\&\qqq\qq\hspace{-0.4pt}
        + (t - t_1) \dual{\hat \bm_h^1 \bar \ptl \lambda_h^1 + \bar \ptl \hat \bm_h^1 \lambda_h^1, \bphi}
        + \dual{\hat \bm_h^1 \lambda_h^1 - \hat \bm_h^{1/2} \lambda_h^{1/2}, \bphi}.
    \end{align*}
    Moreover, by \eqref{EQ: halfstep rewrite}, adding and subtracting first $\dual{\hat \bm_h^{1} \lambda_h^{1/2}, \bphi}$ and second $\frac{\tau}{2} \dual{\bar \ptl \hat \bm_h^{1} \lambda_h^{1}, \bphi}$, we obtain
    \begin{align*}
        \dual{\hat \bm_h^1 \lambda_h^1 - \hat \bm_h^{1/2} \lambda_h^{1/2}, \bphi}
        &= 
        \frac{\tau}{2} \dual{\hat \bm_h^1 \bar \ptl \lambda_h^{1}, \bphi}
        + \frac{\tau}{2} \dual{\bar \ptl \hat \bm_h^1 \lambda_h^{1/2}, \bphi}
        \\&=
        \frac{\tau}{2} \dual{\hat \bm_h^1 \bar \ptl \lambda_h^{1} + \bar \ptl \hat \bm_h^1  \lambda_h^{1} , \bphi}
        - \frac{\tau^2}{4} \dual{\bar \ptl \hat \bm_h^1 \bar \ptl \lambda_h^1, \bphi}.
    \end{align*}
    Combining the above equations and using $t - t_1 + \frac{\tau}{2} = t - t_{1/2}$ yields
    \begin{equation}
    \label{EQ: par err eq init rewrite 3.2}
    \begin{aligned}
        \dual{\bm_h \lambda_h - \hat \bm_h^{1/2} \lambda_h^{1/2}, \bphi}
        &= 
        \dual{(\bm_h - \hat \bm_h) \lambda_h^1, \bphi}
        + (t - t_1) \dual{(\bm_h^1 - \hat \bm_h^1) \bar \ptl \lambda_h^1, \bphi} 
        - \frac{\tau^2}{4} \dual{\bar \ptl \hat \bm_h^1 \bar \ptl \lambda_h^1, \bphi}
        \\&\qq\hspace{-0.4pt}
        + (t - t_{1/2}) \dual{\hat \bm_h^1 \bar \ptl \lambda_h^1 + \bar \ptl \hat \bm_h^1 \lambda_h^1, \bphi}
        + (t - t_1)^2\dual{\bar \ptl \bm_h^1 \bar \ptl \lambda_h^1, \bphi}
    \end{aligned}
    \end{equation}
    All in all, we conclude by inserting \eqref{EQ: par err eq init rewrite 3.1} and \eqref{EQ: par err eq init rewrite 3.2} into the right-hand side of \eqref{EQ: par err eq init 3}
    \begin{equation}
    \label{EQ: parabolic err. eq. init}
    \begin{split}
            &{\alpha}\dual{\ptl_t \errM, \bphi} + \dual{\bm \times \ptl_t \errM, \bphi} + \dual{\errm \times \ptl_t \bM_h, \bphi} + a(\errw, \bphi) + b^{\bm}(\errlam, \bphi) + b^{\errm}(\lambda_h, \bphi)
        \\& \qq=
            \dual{\bff_h - \bff, \bphi} + \dual{(\bm_h^1 - \hat \bm_h^1) \times \ptl_t \bM_h, \bphi}
            + (t-t_1) \dual{(\bar \ptl \bm_h^1 - \bar \ptl \hat \bm_h^1) \times \bar \ptl \bm_h^1, \bphi}
            \\&\qqq
            + (t-t_1)(t-t_{1/2}) \dual{\bar \ptl \bm_h^1 \times \bar \ptl^2 \bm_h^2, \bphi} 
            + \dual{(\bm_h - \hat \bm_h) \lambda_h^1, \bphi}
            \\ &\qqq
            + (t - t_1) \dual{(\bm_h^1 - \hat \bm_h^1) \bar \ptl \lambda_h^1, \bphi}
            - \frac{\tau^2}{4} \dual{\bar \ptl \hat \bm_h^1 \bar \ptl \lambda_h^1, \bphi}
            + (t - t_1)^2 \dual{\bar \ptl \bm_h^1 \bar \ptl \lambda_h^1, \bphi}
            \\ &\qqq
            + \dual{(\bI - \bP_0^1) (\hat \bm_h^{1/2} \times \bar \ptl \bm_h^1), \bphi} + \dual{(\bI - \bP_0^1)(\hat \bm_h^{1/2} \lambda_h^{1/2}), \bphi}
            \\ &\qqq
            + (t - t_{1/2}) \dual{\bPsi_0, \bphi}
        \\&\qq
            = \dual{\br_h^0, \bphi}
    \end{split}
    \end{equation}
    with $\bPsi_0 = \alpha \bar \ptl^2 \bm_h^2 + \bar \ptl \hat \bm_h^1 \times \bar \ptl \bm_h^1 + \hat \bm_h^1 \times \bar \ptl^2 \bm_h^2 - \bar \ptl (\Delta_h^1 \bm_h^1) + \hat \bm_h^1 \bar \ptl \lambda_h^1 + \bar \ptl \hat \bm_h^1 \lambda_h^1 - \bar \ptl \bff_h^1$.
    Following the proof of Theorem~\ref{THM: main theorem} and estimating
    \begin{align*}
        \norm{\grad (\bW - \bm)(0)} \le \norm{\grad (\bW - \bM_h)(0)} + \norm{\grad( \bP_0^0 \bm(0) - \bm(0))}
    \end{align*}
    concludes the proof.
\end{proof}

In the following Lemma, we estimate the difference between the linear reconstruction $\bm_h$ and the three-point reconstruction $\bM_h$.

\begin{Lemma}
    \label{LEM: Mh - mh}
    Let $\bm_h^n$ be the discrete solution of \eqref{EQ: discrete LLG} and $\bM_h$ be the corresponding three-point reconstruction defined in \eqref{EQ:def-Mh_uni} - \eqref{EQ: initial tp-reconstruction}. Then, for $t \in [0, t_N]$ it holds
    \begin{align*}
        \norm{\grad (\bM_h - \bm_h)(t)}^2 
        \lesssim 
        \tau^4 \max\limits_{2 \le n \le N} \norm{\grad \bar \ptl^2_n \bm_h^n}^2 + \tau^4 \max\limits_{2 \le n \le N} \norm{\grad (\bar \ptl^2 - \bar \ptl^2_n) \bm_h^n}^2 = \tau^4 \EE_1 + \Xi_2,
    \end{align*}
    where the estimators are defined in Definition~\ref{DEF: error est.}.
\end{Lemma}
\begin{proof}
    The proof follows immediately by estimating the difference of the reconstruction \eqref{EQ:def-Mh_uni} and \eqref{EQ: initial tp-reconstruction} using the triangle inequality.
    By definition, we have
    \begin{equation}
    \begin{aligned}
        \label{EQ: bM_h - bm_h}
        \bM_h(t) - \bm_h(t) &= \frac12 (t - t_n) (t - t_{n-1}) \bar \ptl^2 \bm_h^n, &&t \in I_n, n \ge 2,
        \\
        \bM_h(t) - \bm_h(t) &= \frac12 (t - t_0) (t - t_1) \bar \ptl^2 \bm_h^2, &&t \in I_1.
    \end{aligned}
    \end{equation}
    Thus, by the triangle inequality and \eqref{EQ: bM_h - bm_h}, we obtain
    \begin{align*}
        \norm{\grad (\bM_h - \bm_h)(t)}^2 
        \lesssim \tau^4 \norm{\grad \bar \ptl^2 \bm_h^n}^2
        \le 
        \tau^4 \norm{\grad \bar \ptl^2_n \bm_h^n}^2 + \tau^4 \norm{\grad (\bar \ptl^2 - \bar \ptl^2_n) \bm_h^n}^2
    \end{align*}
    for $t \in I_n$, $n \ge 2$ and $\norm{\grad (\bM_h - \bm_h)(t)}^2 \lesssim \tau^4 \norm{\grad \bar \ptl^2 \bm_h^2}^2$ for $t \in I_1$.
\end{proof}

Next, we derive error bounds for the error between the time-space reconstruction $\bW$ and the temporal reconstructions, which result in the space error estimators.

\begin{Lemma}
    \label{LEM: apost dt (W-M)}
    Let $\bm_h^n$ be the solution of \eqref{EQ: discrete LLG}, $\bM_h$ be the corresponding three-point reconstructions defined in \eqref{EQ:def-Mh_uni} - \eqref{EQ: initial tp-reconstruction} and $\bW$ be the three-point time-space reconstruction defined in \eqref{EQ: three-point time-space reconstruction}. 
    Then, for $t \in [0, t_N]$ it holds
    \begin{align}
        \norm{\grad (\bW - \bM_h)(t)}^2 &\lesssim 
        \Lambda_1, \label{EQ: apodt W-M 1.}
        \\
        \int_{0}^{t_N} \norm{\ptl_t (\bW - \bM_h)}_{\bHk{1}}^2 \dt 
        & \lesssim \Lambda_2 + \Xi_1, \label{EQ: apodt W-M 2.}
        \\
        \int_{0}^{t_N} \norm{(\bW - \bm_h)}^2 \dt 
        &\lesssim \Lambda_3 + \EE_6, \label{EQ: apodt W-M 3.}
    \end{align}
    where the estimators are defined in Definition~\ref{DEF: error est.}.
\end{Lemma}
\begin{proof}
    Recall that $\ptl_t \bM_h(t) = \bar \ptl^B \bm_h^n + (t- t_n) \bar \ptl^2 \bm_h^n$ and $\ptl_t \bW(t) = \bar \ptl^B \bRR \bm_h^n + (t- t_n) \bar \ptl^2 \bRR \bm_h^n$ for $t \in I_n$ and $n \ge 2$, as well as $\bar \ptl^B (\cdot) = \frac{\tau}{2} \bar \ptl^2 (\cdot) + \bar \ptl(\cdot)$.
    Using the definition of the three-point reconstruction \eqref{EQ:def-Mh_uni} for $t\in I_n$, $n \ge 2$, along with the discrete differentials given in \eqref{EQ: Def ptl bar k} and \eqref{EQ: Def ptlB}, we obtain
    \begin{equation}
    \label{EQ: W-M In}
    \begin{aligned}
        \ptl_t (\bW - \bM_h) &= \bar \ptl^B (\bm_h^n - \bRR \bm_h^n) + (t-t_n) \bar \ptl^2 (\bm_h^n - \bRR \bm_h^n)
        \\ &= \bigg(t-t_n + \frac{\tau}{2}\bigg) \bar \ptl^2 (\bm_h^n - \bRR \bm_h^n) + \bar \ptl (\bm_h^n - \bRR \bm_h^n)
        \\ &= \bigg(\frac{t-t_n}{\tau} + \frac{3}{2}\bigg) \bar \ptl (\bm_h^n - \bRR \bm_h^n)
        - \bigg(\frac{t-t_n}{\tau} + \frac{1}{2}\bigg) \bar \ptl (\bm_h^{n-1} - \bRR \bm_h^{n-1}).
    \end{aligned}
    \end{equation}
    Similarly, applying the definition of the three-point reconstruction \eqref{EQ: initial tp-reconstruction} for $t \in I_1$ and \eqref{EQ: Def ptl bar k} yields
    \begin{equation}
    \label{EQ: W-M I1}
    \begin{aligned}
        \ptl_t (\bW - \bM_h) &= \bar \ptl (\bm_h^1 - \bRR \bm_h^1) + (t-t_{1/2}) \bar \ptl^2 (\bm_h^2 - \bRR \bm_h^2)
        \\ &= \frac{t-t_{1/2}}{\tau} \, \bar \ptl (\bm_h^2 - \bRR \bm_h^2)
        + \bigg(1 - \frac{t-t_{1/2}}{\tau} \bigg) \bar \ptl (\bm_h^1 - \bRR \bm_h^1).
    \end{aligned}
    \end{equation}
    Since $|t-t_n| \le \tau$ for $t \in I_n$, $n \ge 2$, and $|t-t_{1/2}| \le \tau$ for $t \in I_1$, we obtain with \eqref{EQ: W-M In} and \eqref{EQ: W-M I1}
    \begin{align*}
        \int_0^T \norm{\ptl_t (\bW - \bM_h)}_{\bHk{1}}^2 \dt 
        &\lesssim
        \sum\limits_{n=1}^N \tau \norm{\bar \ptl (\bm_h^n - \bRR \bm_h^n)}_{\bHk{1}}^2.
    \end{align*}
    Finally, adding and subtracting $\bRR \bar \ptl_n \bm_h^n$ inside $\norm{\cdot}_{\bHk{1}}$, using $\bar \ptl \bRR \bm_h^n = \bRR \bar \ptl \bm_h^n$ and the triangle inequality leads to
    \begin{align}
        \label{EQ: ptlt W - Mh eq}
        \int_0^T \norm{\ptl_t (\bW - \bM_h)}_{\bHk{1}}^2 \dt 
        & \lesssim
        \sum\limits_{n=1}^N \tau \norm{\bar \ptl_n \bm_h^n - \bRR \bar \ptl_n \bm_h^n}_{\bHk{1}}^2
        + \tau \norm{(\bI - \bRR)(\bar \ptl - \bar \ptl_n) \bm_h^n}_{\bHk{1}}^2.
    \end{align}
    Estimating $\norm{\bar \ptl_n \bm_h^n - \bRR \bar \ptl_n \bm_h^n}_{\bHk{1}}^2\lesssim \eta(\bar \ptl_n \bm_h^n; \Hsym^1)$ 
    via \eqref{EQ: ell. rec. a post} concludes \eqref{EQ: apodt W-M 2.}.

    Let $\star \in \{\bLsym^2, \bHk{1}\}$ and $\norm{\cdot}_\star$ denote the norm under consideration.
    Using the definition of the reconstructions, we have
    \begin{align}
        \label{EQ: bW - bM_h}
        \norm{(\bW - \bM_h)(t)}_\star &\lesssim \norm{\bRR \bm_h^n - \bm_h^n}_\star + \norm{\bRR \bm_h^{n-1} - \bm_h^{n-1}}_\star + \norm{\bRR \bm_h^{n-2} - \bm_h^{n-2}}_\star
    \end{align}
    for $t \in I_n$, $n \ge 2$, and 
    \begin{align*}
        \norm{(\bW - \bM_h)(t)}_\star &\lesssim \norm{\bRR \bm_h^2 - \bm_h^2}_\star + \norm{\bRR \bm_h^{1} - \bm_h^{1}}_\star + \norm{\bRR \bm_h^{0} - \bm_h^{0}}_\star
    \end{align*}
    for $t \in I_1$.
    Combining the above equations with \eqref{EQ: bM_h - bm_h} and \eqref{EQ: ell. rec. a post}, we arrive at
    \begin{align*}
        \int_{0}^{t_N} \norm{(\bW - \bm_h)}^2 \dt 
        &\lesssim
        \int_{0}^{t_N} \big( \norm{\bW - \bM_h}^2 + \norm{\bM_h - \bm_h}^2 \big) \dt 
        \\&\lesssim
        \sum\limits_{n=0}^N \tau \eta(\bm_h^n; \Lsym^2) + \sum\limits_{n=2}^N \tau^5 \norm{\bar \ptl^2 \bm_h^n}^2.
    \end{align*}
    This establishes \eqref{EQ: apodt W-M 3.}. 
    Finally, using \eqref{EQ: bW - bM_h} and estimating $\norm{\grad (\bm_h^n - \bRR \bm_h^n)}^2 \lesssim \eta(\bm_h^n; \Hsym^1)$ with \eqref{EQ: ell. rec. a post} concludes \eqref{EQ: apodt W-M 1.}.
\end{proof}

Below, we bound the difference between the linear and quadratic time-space reconstruction similar to Lemma~\ref{LEM: Mh - mh}.

\begin{Lemma}
    \label{LEMMA: laplh w - W}
    Let $\bm_h^n$ be the solution of \eqref{EQ: discrete LLG} and $\bw$, $\bW$ be the corresponding time-space reconstructions defined in \eqref{EQ: elliptic temp reconstruction} - \eqref{EQ: three-point time-space reconstruction}. Then
    \begin{align*}
        \int_{0}^{t_N}  \norm{\Delta_h (\bw - \bW)}^2 \dt 
        &\lesssim
        \sum\limits_{n = 2}^N \tau^5 \norm{\bar \ptl^2 (- \Delta_h^n)\bm_h^n}^2
        =
        \tau^4 \EE_5 
    \end{align*}
    where the estimator is defined in Definition~\ref{DEF: error est.}.
\end{Lemma}
\begin{proof}
    Recall that 
    \begin{equation*}
    \begin{aligned}
        - \Delta_h (\bw - \bW) &=  \frac12 (t-t_n)(t - t_{n-1}) \bar\ptl^2 (- \Delta_h^n) \bm_h^n, &&t \in I_n, n \ge 2,
        \\
        - \Delta_h (\bw - \bW) &=  \frac12 (t-t_0)(t - t_{1}) \bar\ptl^2 (- \Delta_h^2) \bm_h^2, &&t \in I_1.
    \end{aligned}
    \end{equation*}
    Thus
    \begin{align*}
        \int_0^T \norm{\Delta_h (\bw - \bW)}^2 \dt 
        &\lesssim 
        \sum\limits_{n = 2}^N \tau^5 \norm{ \bar \ptl^2 (- \Delta_h^n) \bm_h^n}^2.
    \end{align*}
\end{proof}

Next, we compute bounds for the residuals of the parabolic error equations     \eqref{EQ:parabolic err eq reconstruction} and \eqref{EQ: parabolic err. eq. init}.

\begin{Lemma}
    \label{LEM: apost rh and rh0}
    Let $\br_h$ be defined by \eqref{EQ: residual} and $\br_h^0$ by \eqref{EQ: r_h^i}.
    Then,
    \begin{align*}
        \int_0^{t_1} \norm{\br_h^0}^2 \dt &\lesssim \FF_1 + \CC_1 + \QQ_1 + \II_2 + \II_3,
        \\
        \int_{t_1}^{t_N} \norm{\br_h}^2 \dt 
        &\lesssim 
        \FF_2 + \tau^4 (\EE_2 + \EE_3 + \EE_4) + \Xi_3 + \CC_2 + \QQ_2 + \II_1,
    \end{align*}
    where the estimators are defined in Definition~\ref{DEF: error est.}.
\end{Lemma}
\begin{proof}
    Recall that
    \begin{align*}
        \dual{\br_h, \bphi}
        &=
        \dual{\bff_h^n - \bff, \bphi} + {\alpha} \dual{(\bar \ptl^B - \bar \ptl_n^B) \bm_h^n, \bphi} 
            + \dual{(\bm_h^n - \hat \bm_h^n) \times \ptl_t \bM_h, \bphi}
            \\&\qqq
            + (t-t_n) \dual{(\bar \ptl \bm_h^n - \bar \ptl \hat \bm_h^n) \times \bar \ptl^B \bm_h^n, \bphi}
            + (t-t_n)^2 \dual{\bar \ptl \bm_h^n \times \bar \ptl^2 \bm_h^n, \bphi} 
            \\ &\qqq
            + \dual{(\bm_h - \hat \bm_h) \lambda_h^n, \bphi}
            + (t - t_n) \dual{(\bm_h^n - \hat \bm_h^n) \bar \ptl \lambda_h^n, \bphi}
            \\&\qqq
            + (t - t_n)^2 \dual{\bar \ptl \bm_h^n \bar \ptl \lambda_h^n, \bphi}
            + (t - t_n) \dual{\bPsi, \bphi}
            \\ &\qqq
            + \dual{(\bI - \bP_0^n) (\hat \bm_h^n \times \bar \ptl^B \bm_h^n), \bphi} + \dual{(\bI - \bP_0^n)(\hat \bm_h^n \lambda_h^n), \bphi},
    \end{align*}
    for all $\bphi \in \Lsym^2(0, t_1; \bV)$, where $\bPsi$ is given by \eqref{EQ: Psi},
    for $t \in I_n$ and $n \ge 2$.
    Although  $\bPsi$ is already computable a posteriori, we further simplify it to clarify different error components and obtain more manageable terms.
    We start by using equation \eqref{EQ:pointwise LLG},~i.e.,~
    \begin{align*}
        &{\alpha} \dual{\bar \ptl_n^B \bm_h^n + \bP_0^n (\hat \bm_h^n \times \bar \ptl^B \bm_h^n) - \Delta_h^n \bm_h^n + \bP_0^n (\hat \bm_h^n \lambda_h^n) - \bff_h^n, \bphi}
        = 0, \qq \forall \bphi \in \Lsym^2(0, t_1; \bV),
    \end{align*}
    to obtain
    \begin{align*}
        \dual{\bPsi, \bphi} &=
            \dual{{\alpha} \bar \ptl^2 \bm_h^n + \bar \ptl \hat \bm_h^n \times \bar \ptl^B \bm_h^n + \hat \bm_h^n \times \bar \ptl^2 \bm_h^n + \hat \bm_h^n \bar \ptl \lambda_h^n + \bar \ptl \hat \bm_h^n \lambda_h^n, \bphi}
            \\&\quad + \frac1\tau \dual{- {\alpha} \bar \ptl_n^B \bm_h^n - \bP_0^n (\hat \bm_h^n \times \bar \ptl^B \bm_h^n) - \bP_0^n (\hat \bm_h^n \lambda_h^n) + \bff_h^n, \bphi}
            \\&\quad + \frac1\tau  \dual{{\alpha} \bar \ptl_{n-1}^B \bm_h^{n-1} + \bP_0^{n-1} (\hat \bm_h^{n-1} \times \bar \ptl^B \bm_h^{n-1}) + \bP_0^{n-1} (\hat \bm_h^{n-1} \lambda_h^{n-1}) - \bff_h^{n-1}, \bphi} 
        \\&=
            \dual{\bar \ptl \bff_h^n, \bphi}
            + {\alpha} \dual{\bar \ptl^2 \bm_h^n - \frac{1}{\tau}( \bar \ptl_n^B \bm_h^n - \bar \ptl_{n-1}^B \bm_h^{n-1}), \bphi}
            + \dual{\hat \bm_h^n \bar \ptl \lambda_h^n + \bar \ptl \hat \bm_h^n \lambda_h^n, \bphi} 
            \\&\quad
            - \dual{\bar \ptl \big(\bP_0^n (\hat \bm_h^n \lambda_h^n)\big), \bphi}
            + \dual{\bar \ptl \hat \bm_h^n \times \bar \ptl^B \bm_h^n + \hat \bm_h^n \times \bar \ptl^2 \bm_h^n - \bar \ptl \big(\bP_0^n(\hat \bm_h^n \times \bar \ptl^B \bm_h^n)\big), \bphi}  
    \end{align*}
    for $t \in I_n$, $n \ge 3$. 
    Since $\bar \ptl_n^B \bm_h^n = \frac{\tau}{2} \bar \ptl_n^2 \bm_h^n + \bar \ptl_n \bm_h^n$ by \eqref{EQ: Def ptlB}, we have
    \begin{equation}
    \label{EQ: transform ptlB}
    \begin{aligned}
        \bar \ptl^2 \bm_h^n - \frac{1}{\tau} \big( \bar \ptl^B_n \bm_h^n - \bar \ptl^B_{n-1} \bm_h^{n-1} \big)
        &= \bar \ptl^2 \bm_h^n - \frac{\bar \ptl_n \bm_h^n - \bar \ptl_{n-1} \bm_h^{n-1}}{\tau} - \frac12 \big( \bar \ptl^2_n \bm_h^n - \bar \ptl^2_{n-1} \bm_h^{n-1} \big)
        \\ &=
        \bar \ptl (\bar \ptl - \bar \ptl_n ) \bm_h^n 
        - \frac12 \big( \bar \ptl^2_n \bm_h^n - \bar \ptl^2_{n-1} \bm_h^{n-1} \big).
    \end{aligned}
    \end{equation}
    Applying \eqref{EQ: transform ptlB} and adding and subtracting $\dual{\bar \ptl (\hat \bm_h^n \lambda_h^n) + \bar \ptl (\hat \bm_h^n \times \bar \ptl^B \bm_h^n), \bphi}$ leads to
    \begin{equation}
    \label{EQ: rewrite psi phi}
    \begin{split}\raisetag{-9ex}
        \dual{\bPsi, \bphi} &=
        \dual{\bar \ptl \bff_h^n, \bphi}
        - \frac{{\alpha}}2 \dual{\bar \ptl_n^2 \bm_h^n - \bar \ptl_{n-1}^2 \bm_h^{n-1}, \bphi} + {\alpha} \dual{\bar \ptl \big( (\bar \ptl - \bar \ptl_n) \bm_h^n\big), \bphi} 
        \\&\quad
        + \dual{\bar \ptl \big((\bI - \bP_0^n)(\hat \bm_h^n \times \bar \ptl^B \bm_h^n)\big), \bphi}
        + \dual{\hat \bm_h^n \bar \ptl \lambda_h^n + \bar \ptl \hat \bm_h^n \lambda_h^n - \bar \ptl (\hat \bm_h^n \lambda_h^n), \bphi}
        \\&\quad
        + \dual{\bar \ptl \hat \bm_h^n \times \bar \ptl^B \bm_h^n + \hat \bm_h^n \times \bar \ptl^2 \bm_h^n - \bar \ptl (\hat \bm_h^n \times \bar \ptl^B \bm_h^n), \bphi} 
        + \dual{\bar \ptl \big((\bI -\bP_0^n) (\hat \bm_h^n \lambda_h^n)\big), \bphi}.
    \end{split}
    \end{equation}
    Using that 
    \begin{align*}
        \bar \ptl (\hat \bm_h^n \times \bar \ptl^B \bm_h^n)
        = 
        \bar \ptl \hat \bm_h^n \times \bar \ptl^B \bm_h^n + \hat \bm_h^n \times \bar \ptl (\bar \ptl^B \bm_h^n) - \tau \bar \ptl \hat \bm_h^n \times \bar \ptl (\bar \ptl^B \bm_h^n)
    \end{align*}
    yields
    \begin{equation}
    \label{EQ: rh apost cross}
    \begin{aligned}
        &\bar \ptl \hat \bm_h^n \times \bar \ptl^B \bm_h^n + \hat \bm_h^n \times \bar \ptl^2 \bm_h^n - \bar \ptl (\hat \bm_h^n \times \bar \ptl^B \bm_h^n)
        \\&\qq=
        \bar \ptl \hat \bm_h^n \times \bar \ptl^B \bm_h^n + \hat \bm_h^n \times \bar \ptl^2 \bm_h^n
        \\&\qqq- \bar \ptl \hat \bm_h^n \times \bar \ptl^B \bm_h^n - \hat \bm_h^n \times \bar \ptl (\bar \ptl^B \bm_h^n) + \tau \bar \ptl \hat \bm_h^n \times \bar \ptl (\bar \ptl^B \bm_h^n)
        \\&\qq=
        - \frac\tau2  \hat \bm_h^n \times \bar \ptl^3 \bm_h^n 
        + \frac{\tau^2}2 \bar \ptl \hat \bm_h^n \times \bar \ptl^3 \bm_h^n
        + \tau \bar \ptl \hat \bm_h^n \times \bar \ptl^2 \bm_h^n
        \\&\qq=
        - \frac{\tau}2 \hat \bm_h^{n-1} \times \bar \ptl^3 \bm_h^n
        + \tau \bar \ptl \hat \bm_h^n \times \bar \ptl^2 \bm_h^n.
    \end{aligned}
    \end{equation}
    Moreover, we have
    \begin{align*}
        \bar \ptl (\hat \bm_h^n \lambda_h^n) &= 
        \bar \ptl \hat \bm_h^n \lambda_h^n + \hat \bm_h^{n-1} \bar \ptl \lambda_h^n 
        =
        \bar \ptl \hat \bm_h^n \lambda_h^n + \hat \bm_h^n \bar \ptl \lambda_h^n - \tau \bar \ptl \hat \bm_h^n \bar \ptl \lambda_h^n.
    \end{align*}
    Combining the above equations with \eqref{EQ: rewrite psi phi}, we conclude
    \begin{align*}
        \dual{\bPsi, \bphi} &=
            \dual{\bar \ptl \bff_h^n, \bphi}
            - \frac{{\alpha}}2 \dual{ \bar \ptl_n^2 \bm_h^n - \bar \ptl_{n-1}^2 \bm_h^{n-1}, \bphi} + {\alpha} \dual{\bar \ptl \big( (\bar \ptl - \bar \ptl_n) \bm_h^n\big), \bphi} + \tau \dual{\bar \ptl \hat \bm_h^n \bar \ptl \lambda_h^n, \bphi} 
            \\&\quad
            + \dual{\bar \ptl \big((\bI - \bP_0^n)(\hat \bm_h^n \times \bar \ptl^B \bm_h^n)\big), \bphi}
            + \dual{\bar \ptl \big((\bI -\bP_0^n) (\hat \bm_h^n \lambda_h^n)\big), \bphi}
            \\&\quad
            - \frac{\tau}2 \dual{  \hat \bm_h^{n-1} \times \bar \ptl^3 \bm_h^n , \bphi} 
            + \tau \dual{\bar \ptl \hat \bm_h^n \times \bar \ptl^2 \bm_h^n, \bphi} \eqqcolon g^n(t).            
    \end{align*}
    For the case $t \in I_2$, we have due to the definition of $\bm_h^{-1}$ given by \eqref{EQ: mh -1} and the trapezoidal scheme \eqref{EQ: init trapezoidal m1}
    \begin{equation}
    \begin{aligned}
            {\alpha} \bar \ptl^B \bm_h^1 
        &= 
            \frac{{\alpha}}{\tau} \big( \frac32 \bm_h^1 - 2 \bm_h^0 + \frac12 \bm_h^{-1} \big)
        \\&=
            2 {\alpha} \bar \ptl \bm_h^1 - \bff_h^0 - \Delta_h^0 \bm_h^0 + \bP_0^1(\hat \bm_h^0 \times \bar \ptl \bm_h^1) + \bP_0^1 (\lambda_h^0 \hat \bm_h^0)
        \\ &=  
            \bff_h^1 + \Delta_h^1\bm_h^1 - \bP_0^1(\hat \bm_h^1 \times \bar \ptl \bm_h^1) + \bP_0^1 (\lambda_h^0 \hat \bm_h^0) - 2 \bP_0^1(\lambda_h^{1/2} \hat \bm_h^{1/2}).
    \end{aligned}
    \end{equation}
    Using that 
    $\lambda_h^0 \hat \bm_h^0 + \lambda_h^1 \hat \bm_h^1 - 2 \lambda_h^{1/2} \hat \bm_h^{1/2} = \frac12 \tau^2 \bar \ptl \lambda_h^1 \bar \ptl \hat \bm_h^1$,
    we get 
    \begin{equation}
    \label{EQ: bar ptl^B bm_h^1}
    \begin{aligned}        
        {\alpha} \bar \ptl^B \bm_h^1 
        &= 
        \bff_h^1 + \Delta_h^1\bm_h^1 - \bP_0^1(\hat \bm_h^1 \times \bar \ptl \bm_h^1) - \bP_0^1 (\lambda_h^1 \hat \bm_h^1) + \frac12 \tau^2  \bP_0^1(\bar \ptl \lambda_h^1 \bar \ptl \hat \bm_h^1)
        \\ &= 
        \bff_h^1 + \Delta_h^1\bm_h^1 - \bP_0^1(\hat \bm_h^1 \times \bar \ptl^B \bm_h^1) - \bP_0^1 (\lambda_h^1 \hat \bm_h^1) + \frac12 \tau^2  \bP_0^1(\bar \ptl \lambda_h^1 \bar \ptl \hat \bm_h^1) 
        \\&\qq+ \frac{\tau}{2} \bP_0^1 (\hat \bm_h^1 \times \bar \ptl^2 \bm_h^1),
    \end{aligned}
    \end{equation}
    where 
    ${\alpha} \bar \ptl^2 \bm_h^1 = \frac{{\alpha}}{\tau^2} \big( \bm_h^1 - 2 \bm_h^0 + \bm_h^{-1} \big)$ with $\bm_h^{-1}$ given by \eqref{EQ: mh -1}.
    Thus, we obtain for $t \in I_2$ using \eqref{EQ: bar ptl^B bm_h^1}
    \begin{equation*}       
        \begin{aligned}
        \dual{\bPsi, \bphi} =
            g^2(t) + \frac\tau2 \dual{\bP_0^1(\bar \ptl \lambda_h^1 \bar \ptl \hat \bm_h^1)} + \frac12  \dual{ \bP_0^1( \hat \bm_h^1 \times \bar \ptl^2 \bm_h^1), \bphi}. 
        \end{aligned}
    \end{equation*} 
    
    Finally, the triangle inequality yields for the residual $\br_h$ defined in \eqref{EQ: residual}
    \begin{align*}
        \norm{\br_h} &\lesssim
        \norm{\bff_h - \bff} 
        + \norm{(\bar \ptl^B - \bar \ptl_n^B) \bm_h^n}
        + \norm{(\bm_h^n - \hat \bm_h^n) \times \ptl_t \bM_h}
        \\ &\qq 
        + \tau \norm{(\bar \ptl \bm_h^n - \bar \ptl \hat \bm_h^n) \times \bar \ptl^B \bm_h^n}
        + \tau^2 \norm{\bar \ptl \bm_h^n \times \bar \ptl^2 \bm_h^n}
        \\ &\qq 
        + \norm{(\bm_h - \hat \bm_h) \lambda_h^n}
        + \tau \norm{(\bm_h^n - \hat \bm_h^n)\bar \ptl \lambda_h^n}
        + \tau^2 \norm{\bar \ptl \bm_h^n \bar \ptl \lambda_h^n}
        \\ &\qq
        + \norm{(\bI - \bP_0^n) (\hat \bm_h^n \times \bar \ptl^B \bm_h^n)} 
        + \norm{(\bI - \bP_0^n)(\hat \bm_h^n \lambda_h^n)}
        \\&\qq
        + \tau \norm{\bar \ptl_n^2 \bm_h^n - \bar \ptl_{n-1}^2 \bm_h^{n-1}}
        + \tau \norm{\bar \ptl \big( (\bar \ptl - \bar \ptl_n) \bm_h^n\big)}
        + \tau^2 \norm{\bar \ptl \hat \bm_h^n \bar \ptl \lambda_h^n}
        \\ &\qq 
        + \tau \norm{\bar \ptl \big((\bI - \bP_0^n)(\hat \bm_h^n \times \bar \ptl^B \bm_h^n)\big)}
        + \tau \norm{\bar \ptl \big( (\bI - \bP_0^n)(\hat \bm_h^n \lambda_h^n)\big)}
        \\&\qq
        + \tau^2 \norm{\hat \bm_h^{n-1} \times \bar \ptl^3 \bm_h^n}
        + \tau^2 \norm{\bar \ptl \hat \bm_h^n \times \bar \ptl^2 \bm_h^n}
        \\ &\qq 
        + \tau^2 \chi_{I_2}(t) \norm{\bP_0^1(\bar \ptl \lambda_h^1 \bar \ptl \hat \bm_h^1)}
        + \tau \chi_{I_2}(t) \norm{\bP_0^1(\hat \bm_h^1 \times \bar \ptl^2 \bm_h^1)},
    \end{align*}
    for $t \in I_n$ and $n \ge 2$,
    where $\chi_{I_2}$ denotes the indicator function onto $I_2$. Integrating from $t = t_{k-1}$ to $t_k$ and summing over $k=2$ to $N$ concludes the proof for $\br_h$.

    Using the definition of $\br_h^0$ \eqref{EQ: r_h^i}, we obtain
    \begin{align*}
        \norm{\br_h^0} 
        &\lesssim 
        \norm{\bff_h - \bff}
        + \norm{(\bm_h^1 - \hat \bm_h^1) \times \ptl_t \bM_h}
        + \tau \norm{(\bar \ptl \bm_h^1 - \bar \ptl \hat \bm_h^1) \times \bar \ptl \bm_h^1}
        \\ &\qq 
        + \tau^2 \norm{\bar \ptl \bm_h^1 \times \bar \ptl^2 \bm_h^2}
        + \norm{(\bm_h - \hat \bm_h) \lambda_h^1}
        + \tau \norm{(\bm_h^1 - \hat \bm_h^1) \bar \ptl \lambda_h^1}
        + \tau^2 \norm{\bar \ptl \hat \bm_h^1 \bar \ptl \lambda_h^1}
        \\ &\qq
        + \tau^2 \norm{\bar \ptl \bm_h^1 \bar \ptl \lambda_h^1}
        + \norm{(\bI - \bP_0^1)(\hat \bm_h^{1/2} \times \bar \ptl \bm_h^1)}
        + \norm{(\bI - \bP_0^1)(\hat \bm_h^{1/2} \lambda_h^{1/2})} 
        + \tau \norm{\bPsi_0}.
    \end{align*}
\end{proof}

We proceed with the $\Lsym^\infty$-error bounds for the error between the time-space and temporal three-point reconstruction.

\begin{Lemma}
   \label{LEM: apost linfty W-M}
   Let $\bm_h^n$ be the solution of \eqref{EQ: discrete LLG}, $\bM_h$ be the corresponding three-point reconstructions defined in \eqref{EQ:def-Mh_uni} - \eqref{EQ: initial tp-reconstruction} and $\bW$ be the three-point time-space reconstruction defined in \eqref{EQ: three-point time-space reconstruction}. 
    Then
    \begin{align}
        \label{EQ: Linfty1}
        \norm{\bW - \bM_h}_{\Linf(0, T; \bLinf)}^2 &\lesssim \max\limits_{0 \le n \le N} \etainfty{\bm_h^n},
        \\
        \label{EQ: Linfty2}
        \norm{\ptl_t (\bW - \bM_h)}_{\Linf(0, T; \bLinf)}^2 &\lesssim \max\limits_{1 \le n \le N} \etainfty{\bar \ptl_n \bm_h^n} + \max\limits_{1 \le n \le N} \etainfty{(\bar \ptl - \bar \ptl_n) \bm_h^n},
        \\
        \label{EQ: Linfty3}
        \norm{\grad(\bW - \bM_h)}_{\Linf(0, T; \bLinf)}^2 &\lesssim \max\limits_{0 \le n \le N} \etaWinfty{\bm_h^n},
        \\
        \label{EQ: Linfty4}
        \norm{\grad \ptl_t (\bW - \bM_h)}_{\Linf(0, T; \bLinf)}^2 &\lesssim \max\limits_{1 \le n \le N} \etaWinfty{\bar \ptl_n \bm_h^n} + \max\limits_{1 \le n \le N} \etaWinfty{(\bar \ptl - \bar \ptl_n) \bm_h^n},
    \end{align}
    where we employ the estimators from Assumption~\ref{AS: Linfty estimators}.
\end{Lemma}
\begin{proof}
    Combining \eqref{EQ: bW - bM_h} and \eqref{EQ: apost ell rec Linfty} yields \eqref{EQ: Linfty1}. Following \eqref{EQ: ptlt W - Mh eq}, we obtain \eqref{EQ: Linfty2}. Finally, \eqref{EQ: Linfty3} and \eqref{EQ: Linfty4} follow in the same way.
\end{proof}

Finally, we compute the a~posteriori bound for the normalization error of the time-space reconstruction.
\begin{Lemma}
   Let $\bm_h^n$ be the solution of \eqref{EQ: discrete LLG}, $\bM_h$ be the corresponding three-point reconstructions defined in \eqref{EQ:def-Mh_uni} - \eqref{EQ: initial tp-reconstruction}
    and $\bW$ be the three-point time-space reconstruction defined in \eqref{EQ: three-point time-space reconstruction}. 
    Then
    \begin{align*}
        &\int_{0}^{t_N} \norm{(\bI - \bP(\bW)) \ptl_t \bW}_{\bHk{1}}^2 \dt
        \\&\qq\lesssim
        \int_{0}^{t_N} \Big( \norm{(\bI - \bP(\bM_h)) \ptl_t \bM_h}_{\bHk{1}}^2 + \norm{\bM_h - \bW}_{\bHk{1}}^2 + \norm{\ptl_t(\bM_h - \bW)}_{\bHk{1}}^2 \Big) \dt,
        \\&\qq\lesssim \PP + \Lambda_2 + \Lambda_3 + \Xi_1,
    \end{align*}
    where the (hidden) constant depends on $\norm{\bM_h}_{\bWkp{1}{\infty}}$, $\norm{\ptl_t \bM_h}_{\bWkp{1}{\infty}}$, $\etaWinfty{\bm_h^n}$
    and
    where the estimators are defined in Definition~\ref{DEF: error est.}.
\end{Lemma}
\begin{proof}   
    By adding and subtracting $(\bI - \bP(\bW)) \ptl_t \bM_h + \bP(\bM_h) \ptl_t \bM_h$ and using the triangle inequality, we obtain
    \begin{align*}
        \norm{(\bI - \bP(\bW)) \ptl_t \bW}_{\bHk{1}} 
        &\le 
        \norm{(\bI - \bP(\bW)) \ptl_t \bM_h}_{\bHk{1}} + \norm{(\bI - \bP(\bW)) \ptl_t (\bW - \bM_h)}_{\bHk{1}}
        \\&\le 
        \norm{(\bI - \bP(\bM_h)) \ptl_t \bM_h}_{\bHk{1}}
        + \norm{(\bP(\bM_h) - \bP(\bW)) \ptl_t \bM_h}_{\bHk{1}}
        \\&\qq+ \norm{(\bI - \bP(\bW)) \ptl_t (\bW - \bM_h)}_{\bHk{1}}.
    \end{align*}
    Since $\bP(\bW) = \bI - \bW \bW^\trp$, we estimate
    \begin{align*}
        \norm{(\bI - \bP(\bW)) \ptl_t (\bW - \bM_h)}_{\bHk{1}} 
        = \norm{\bW \bW^\trp \ptl_t (\bW - \bM_h)}_{\bHk{1}}
        \le \norm{\bW}_{\bWkp{1}{\infty}}^2 \norm{\ptl_t (\bW - \bM_h)}_{\bHk{1}}.
    \end{align*}
    Using the Lipschitz-type bound of the orthogonal projection in Lemma~\ref{LEM: Lipschitz bound}, the a posteriori bound of the elliptic reconstruction \eqref{EQ: ell. rec. a post} and \eqref{EQ: bW - bM_h}, we obtain
    \begin{align*}
        \int_{0}^{t_N} \norm{(\bP(\bM_h) - \bP(\bW)) \ptl_t \bM_h}_{\bHk{1}}^2 \dt &\le \int_{0}^{t_N} \gamma(\bM_h, \ptl_t \bM_h)^2 \norm{\bM_h - \bW}_{\bHk{1}}^2 \dt 
        \lesssim 
        \Lambda_3.
    \end{align*}
    Finally, applying Lemma~\ref{LEM: apost dt (W-M)} to estimate $\norm{\ptl_t (\bW - \bM_h)}_{\bHk{1}}$ 
    and bounding $\norm{\bW}_{\bWkp{1}{\infty}}^2$ by the triangle inequality 
    \begin{align}
        \label{EQ: triangle eq. Linfty apost bound}
        \norm{\bW}_{\bWkp{1}{\infty}} \le \norm{\bW - \bM_h}_{\bWkp{1}{\infty}} + \norm{\bM_h}_{\bWkp{1}{\infty}},
    \end{align}
    where $\norm{\bW - \bM_h}_{\bWkp{1}{\infty}}$ is bounded by \eqref{EQ: Linfty3}, concludes the proof.
\end{proof}

Finally, applying the triangle inequality as in \eqref{EQ: triangle eq. Linfty apost bound} and Lemma~\ref{LEM: apost linfty W-M}, we obtain the a posteriori computable error bounds for the coefficients $\gamma_{0, \bW}$ and $\gamma_{1, \bW}$ defined in \eqref{EQ: def gammas}.

\section{Variable time-steps}\label{SEC:VARIABLE-TS}

The preceding analysis was carried out for constant time-step sizes to simplify the presentation.
Since practical adaptive algorithms rely on variable time-steps, we now extend the \emph{a posteriori} analysis to this setting.
The proofs follow the same arguments as in the uniform time-step case with accordingly adapted \BDF{2} operator \eqref{EQ: Def ptlB} and three-point reconstruction \eqref{EQ:def-Mh_uni} as in \cite[Sct.~7.1]{BaenschBrenner:2019}. 
We therefore restrict ourselves to deriving the modified error indicators and state the resulting a~posteriori estimate.
Throughout this section, we assume a fixed spatial mesh for simplicity. 
The extension to changing meshes is straightforward as in the uniform case.

Let $\kappa_n \coloneqq \frac{\tau_n}{\tau_{n-1}}$ denote the consecutive step-size ratio for $n \ge 1$.
Then, the variable-step \BDF{2} operator is defined by
\begin{align}
\label{EQ: DEF ptlB for variable step}
\bar \ptl^B \bm_h^n
&\coloneqq
\frac1{\tau_n}
\left(
\frac{1+2\kappa_n}{1+\kappa_n}\bm_h^n
-(1+\kappa_n)\bm_h^{n-1}
+\frac{\kappa_n^2}{1+\kappa_n}\bm_h^{n-2}
\right)
=
\bar\ptl\bm_h^n
+\tau_n\frac{\kappa_n}{1+\kappa_n}\bar\ptl^2\bm_h^n .
\end{align}
The corresponding three-point reconstruction is given by
\begin{equation}
\label{EQ: temporal quadratic polynomial variable step}
\begin{aligned}
\bM_h(t)
&\coloneqq
\bm_h(t)
+
\frac{\kappa_n}{1+\kappa_n}
(t-t_n)(t-t_{n-1})
\bar\ptl^2\bm_h^n,
\\
\partial_t\bM_h(t)
&=
\bar\ptl^B\bm_h^n
+
\frac{2\kappa_n}{1+\kappa_n}
(t-t_n)\bar\ptl^2\bm_h^n,
\end{aligned}
\end{equation}
for $t\in I_n=(t_{n-1},t_n]$, $n\ge2$.
On the first interval $I_1$, we use the reconstruction defined in \eqref{EQ: initial tp-reconstruction}.

Following the proof of Lemma~\ref{LEMMA: parabolic err eq} with the modified reconstruction yields the residual
\begin{align*}
\tilde\bPsi
&\coloneqq
\alpha\frac{2\kappa_n}{1+\kappa_n}\bar\ptl^2\bm_h^n
+\bar\ptl\hat\bm_h^n\times\bar\ptl^B\bm_h^n
+\frac{2\kappa_n}{1+\kappa_n}\hat\bm_h^n\times\bar\ptl^2\bm_h^n
-\bar\ptl(\Delta_h^n\bm_h^n)
+\hat\bm_h^n\bar\ptl\lambda_h^n
+\bar\ptl\hat\bm_h^n\lambda_h^n.
\end{align*}

To derive the modified estimator $\tilde\EE_2$, we rewrite \eqref{EQ: transform ptlB} as
\begin{equation}
\label{EQ: rewrite var 1.1}
\begin{aligned}
&\frac{2\kappa_n}{1+\kappa_n}\bar\ptl^2\bm_h^n
-
\frac1{\tau_n}
\left(
\bar\ptl^B\bm_h^n
-
\bar\ptl^B\bm_h^{n-1}
\right)
=
\frac1{\kappa_n}
\left(
-\frac{\kappa_n}{1+\kappa_n}\bar\ptl^2\bm_h^n
+
\frac{\kappa_{n-1}}{1+\kappa_{n-1}}
\bar\ptl^2\bm_h^{n-1}
\right).
\end{aligned}
\end{equation}
Proceeding as in the proof of Lemma~\ref{LEM: apost rh and rh0} therefore yields
\[
\tilde\EE_2
\coloneqq
\sum_{n=2}^N
\frac{\tau_n^3}{\kappa_n^2}
\left\|
\frac{\kappa_n}{1+\kappa_n}\bar\ptl^2\bm_h^n
-
\frac{\kappa_{n-1}}{1+\kappa_{n-1}}
\bar\ptl^2\bm_h^{n-1}
\right\|^2.
\]

Similarly, rewriting \eqref{EQ: rh apost cross} using \eqref{EQ: DEF ptlB for variable step} leads to
\begin{equation*}
\begin{aligned}
&\bar\ptl\hat\bm_h^n\times\bar\ptl^B\bm_h^n
+
\frac{2\kappa_n}{1+\kappa_n}\hat\bm_h^n\times\bar\ptl^2\bm_h^n
-
\bar\ptl(\hat\bm_h^n\times\bar\ptl^B\bm_h^n)
\\
&\qquad=
\frac{\kappa_n-1}{\kappa_n+1}
\hat\bm_h^n\times\bar\ptl^2\bm_h^n
-
\tau_n
\frac{\kappa_n}{1+\kappa_n}
\hat\bm_h^{n-1}\times\bar\ptl^3\bm_h^n
+
\tau_n
\bar\ptl\hat\bm_h^n\times\bar\ptl^2\bm_h^n,
\end{aligned}
\end{equation*}
which gives rise to the additional estimator
\[
\tilde\EE_7
\coloneqq
\sum_{n=2}^N
\tau_n^3
\frac{(\kappa_n-1)^2}{(\kappa_n+1)^2}
\|
\hat\bm_h^n\times\bar\ptl^2\bm_h^n
\|^2.
\]

All remaining estimators coincide with their counterparts for uniform time-steps.
For completeness, we collect the resulting estimators in the following definition.

\begin{Definition}[Error estimators for variable time-steps]
	\label{DEF: error est. variable step}
	Let $\seq{\bm_h^n}{0}{N}$, $\seq{\hat \bm_h^n}{0}{N}$, $\seq{\bff_h^n}{0}{N}$, $\seq{\lambda_h^n}{0}{N}$ be sequences with $\bm_h^n, \bff_h^n \in \bV_h^n$, $\lambda_h^n \in \spacelam{n}$ and $\hat \bm_h^n \in V$ for all $n \ge 0$
	and define
	\begin{align}
		\label{EQ: mh -1 tau var}
		{\alpha} \bm_h^{-1} &\coloneqq {\alpha} \bm_h^1 - 2 \tau_1 \big( \bff_h^0 + \Delta_h^0 \bm_h^0 - \bLproj{1}(\hat \bm_h^0 \times \bar \ptl \bm_h^1) - \bLproj{1} (\lambda_h^0 \hat \bm_h^0) \big).
	\end{align} 
	With the definition of $\bM_h$ in \eqref{EQ: temporal quadratic polynomial variable step}, we introduce the \textbf{time error estimators}
	\begin{equation*}
		\begin{aligned}
			\qq&\tilde \EE_1 \coloneqq \max\limits_{2 \le n \le N} \tau_n^4 \norm{\grad \bar \ptl^2_n \bm_h^n}^2,\qq 
			&&
			\tilde \EE_2 \coloneqq \sum\limits_{n=2}^N \frac{\tau_n^3}{\kappa_n^2} \norm{\frac{\kappa_n}{\kappa_n + 1} \bar \ptl^2 \bm_h^n - \frac{\kappa_{n-1}}{\kappa_{n-1} + 1} \bar \ptl^2 \bm_h^{n-1}}^2,
			\\
			\alignedintertext{\centering $\displaystyle 
				\tilde \EE_3 \coloneqq \sum\limits_{n=1}^N \tau_n^5 \big(\norm{\bar \ptl \bm_h^n \bar \ptl \lambda_h^n}^2 + \norm{\bar \ptl \hat \bm_h^n \bar \ptl \lambda_h^n}^2 \big), 
				$}
			\alignedintertext{\centering $\displaystyle \tilde \EE_4\coloneqq \sum\limits_{n = 2}^N \tau_n^5 \Big(\norm{\bar \ptl \bm_h^n \times \bar \ptl^2 \bm_h^n}^2 + \norm{\bar \ptl \hat \bm_h^n \times \bar \ptl^2 \bm_h^n}^2 +\norm{\hat \bm_h^{n-1} \times \bar \ptl^3 \bm_h^n}^2\Big),$}
			&\tilde \EE_6 \coloneqq \sum\limits_{n=2}^N \tau_n^5 \norm{\bar \ptl^2 \bm_h^n}^2,
			&&\tilde \EE_7 \coloneqq \sum\limits_{n=2}^N 
			\tau_n^3 \frac{(\kappa_n - 1)^2}{(\kappa_n + 1)^2} \norm{\hat \bm_h^n \times \bar \ptl^2 \bm_h^n}^2,
		\end{aligned}
	\end{equation*}
	the \textbf{reconstruction error estimator}
	\begin{align*}
		\tilde \EE_5 \coloneqq \sum\limits_{n = 2}^N \tau_n^5 \norm{\bar \ptl^2 (- \Delta_h^n)\bm_h^n}^2,
	\end{align*}
	the \textbf{space error estimators}
	\begin{align*}
		&\tilde \Lambda_2 \coloneqq \sum\limits_{n=1}^N \tau_n \, \eta(\bar\ptl_n \bm_h^n; \Hsym^1),
		\qq \tilde \Lambda_3 \coloneqq \sum\limits_{n=0}^N \tau_n \,\eta(\bm_h^n; \Lsym^2),
	\end{align*}
	where $\eta$ is the elliptic a~posteriori estimator defined in \eqref{EQ: eta estimator standard},
	the \textbf{finite element space conforming estimators}
	\begin{align*}
		&\tilde \CC_2 \coloneqq \sum\limits_{n=2}^N 
		\Big(  
		\tau_n \norm{(\bI - \bLproj{n}) (\hat \bm_h^n \times \bar \ptl^B \bm_h^n)}^2
		+ \tau_n \norm{(\bI - \bLproj{n})(\hat \bm_h^n \lambda_h^n)}^2
		\\
		&\qqq+ \tau_n^3 \norm{\bar \ptl \big((\bI - \bLproj{n}) (\hat \bm_h^n \lambda_h^n) \big)}^2
		+ \tau_n^3 \norm{\bar \ptl \big( (\bI - \bLproj{n}) (\hat \bm_h^n \times \bar \ptl^B \bm_h^n) \big)}^2
		\Big)
	\end{align*}
	and the \textbf{extrapolation error estimators}
	\begin{align*}
		&\tilde \QQ_2 \coloneqq 
		\sum\limits_{n=2}^N \Big( \int_{I_n} \big(\norm{(\bm_h^n - \hat \bm_h^n) \times \ptl_t \bM_h}^2 + \norm{(\bm_h - \hat \bm_h) \lambda_h^n}^2 \big) \dt
		\\&\qqq
		+ \tau_n^3 \norm{(\bar \ptl \bm_h^n - \bar \ptl \hat \bm_h^n) \times \bar \ptl^B \bm_h^n}^2
		+ \tau_n^3 \norm{(\bm_h^n - \hat \bm_h^n)\bar \ptl \lambda_h^n}^2 \Big).
	\end{align*}
\end{Definition}

We obtain the following theorem for variable time-steps (and non-changing mesh).
\begin{theorem}[A posteriori error estimate]
	\label{THM: apost variable tau}
	Let $\bm$ be the exact solution of \eqref{EQ: linearized LLG analytical} and $\bm_h^n$ the solution of \eqref{EQ: discrete LLG} for $n = 2, \dots, N$ and the solution of \eqref{EQ: init trapezoidal m1} for $n=1$. Assume $\tau = \tau_1 = \tau_2$ and $V_h^n = V_h$ for all $n = 0, \dots, N$. Then
	\begin{equation*}   
		\begin{aligned}
			\norm{\grad (\bm - \bm_h)}_{\Linf(t_0, t_N; \bLspace)}^2 &\lesssim 
			\FF_1 + \FF_2 
			+ \tilde \EE_1 + \tilde \EE_2 + \tilde \EE_3  + \tilde \EE_4 + \tilde \EE_5 + \tilde \EE_6 + \tilde \EE_7 
			+ \Lambda_1 + \tilde \Lambda_2 + \tilde \Lambda_3
			\\&\qq
			+ \PP
			+ \CC_1 + \tilde \CC_2 
			+ \QQ_1 + \tilde \QQ_2 
			+ \II_1 + \II_2 + \II_3,
		\end{aligned}
	\end{equation*}
	where the (hidden) constant depends on $\alpha$, $\norm{\bM_h}_{\bWkp{1}{\infty}}$, $\norm{\ptl_t \bM_h}_{\bWkp{1}{\infty}}$, $\etaWinfty{\bm_h^n}$ and $\norm{\lambda_h}_{\Linf}$
	and the estimators are given in Definition~\ref{DEF: error est.} and Definition~\ref{DEF: error est. variable step}.
\end{theorem}

Hence the complete \emph{a~posteriori} error analysis developed above applies equally to variable time-steps without requiring any additional analytical ingredients beyond the modified three-point reconstruction.

\section*{Acknowledgments}
This work is based on the author's doctoral dissertation, completed at Karlsruhe Institute of Technology.
The author is grateful for the valuable supervision of W. Dörfler during various stages of this work. 
This research was funded by the Deutsche Forschungsgemeinschaft (DFG, German  Research Foundation) -- Project-ID 258734477 -- SFB 1173.

\printbibliography
\end{document}